\documentclass[12pt, twoside]{article}
\usepackage{amsmath,amssymb}
\usepackage{amsthm}
\usepackage{graphicx}
\usepackage{epsfig}
\usepackage{cite}
\usepackage{subfigure}
\usepackage{hyperref}
\newtheorem{theorem}{Theorem}[section]
\newtheorem{lemma}{Lemma}[section]

\newtheorem{assumption}{Assumption}[section]

\numberwithin{equation}{section}

\numberwithin{equation}{section} \makeatletter
\begin{document}

\title{\textbf{Hopf Bifurcation for an SIS model with age structure}}
\author{Xiangming Zhang{\small \textsc{$^{a,}$\thanks{%
Research was partially supported by NSFC (Grant No. 11471044) and the
Fundamental Research Funds for the Central Universities.}}} and Zhihua Liu%
{\small \textsc{$^{a,*,}$}}\textsc{\thanks{%
{\small Corresponding author. \newline \indent~~E-mail addresses: xiangmingzhang@mail.bnu.edu.cn (X. Zhang), zhihualiu@bnu.edu.cn (Z. Liu).}}} \\
$^{a}$School of Mathematical Sciences, Beijing Normal University,\\
Beijing, 100875, People's Republic of China}

\date{}
\maketitle
\begin{abstract}
An SIS model is investigated in which the infective individuals are assumed to have an infection-age structure. The model is formulated as an abstract non-densely defined Cauchy problem. We study some dynamical properties of the model by using the theory of integrated semigroup, the Hopf bifurcation theory and the normal form theory for semilinear equations with non-dense domain. Qualitative analysis indicates that there exist some parameter values such that this SIS model has a non-trivial periodic solution which bifurcates from the positive equilibrium. Furthermore, the explicit formulae are given to determine the direction of the Hopf bifurcation and the stability of the bifurcating periodic solutions. Numerical simulations are also carried out to support our theoretical results.


\end{abstract}

\section{Introduction}

\noindent

In population dynamics, the age variable plays a vital role in determining the birth, growth and death rates of the populations and their interactions with each other. Loosely speaking, when age structure is introduced into individual interactions, population dynamical models become considerably complex. In general, the progress of disease propagation and individual interactions are modeled by using an ODE system \cite{LiuLevinIwasa-JMB-1986,LiJunZhaoYulinZhuHuaiping-JMAA-2015}. Recently, as the significance of age structure in populations has become increasingly prevalent, there has been explosively growing literature dealing with all kinds of aspects of interacting populations with age structure \cite{MimmoIannelli-1995,WangLangZou-NARWA-2017,XuZhang-DCDS-B-2016,YangLiZhang-IJB-2016,YangRuanXiao-MBE-2015,TangLiu-AMM-2016,LiuLi-JNS-2015,
WangLiu-JMAA-2012,PierreMagalShiguiRuan-MAMS-2009}.
Wang et al. \cite{WangLangZou-NARWA-2017} explored an age-structured hybrid model with two infection modes. They analysed the relative compactness and persistence of the solution semiflow and treated the global attractor by constructing the proper Lyapunov functions.
Yang et al. \cite{YangLiZhang-IJB-2016} investigated a heroin model with age structure and nonlinear incidence rate and studied the stability of equilibria by employing the Lyapunov functions.
An age-structured virus dynamics model with Beddington-DeAngelis infection function was studied by Yang et al. \cite{YangRuanXiao-MBE-2015}. The authors given the basic reproductive number $\mathcal{R}_{0}$ and treated the global behavior of the model in terms of $\mathcal{R}_{0}$.
In \cite{LiuLi-JNS-2015,TangLiu-AMM-2016,WangLiu-JMAA-2012,PierreMagalShiguiRuan-MAMS-2009,LiuMagalRuan-DCDS-B-2016,
LiuMagalRuan-ZAMP-2011,LiuTangMagal-DCDS-B-2015,FuLiuMagal-2015}, by formulating the age-structured model as a non-densely defined
Cauchy problem, the existence of Hopf bifurcation was considered for some
age-structured models.

As is known to all, normal form theory is rather beneficial in simplifying the forms of equations restricted on the center manifolds in researhing the nonlinear dynamical problems, such as the existence of bifurcations and periodic solutions. Normal form theory has been well-extended to various types of
equations, including ordinary differential equations (Guckenheimer and Holmes \cite{GuckenheimerJHolmesP-1983}, Chow et al.\cite{ChowLiWang-1994}), partial differential equations (Kokubu \cite{KokubuHiroshi-JJAM-1984}, Eckmann et al. \cite{EckmannEpsteinWayne-AIHPPT-1993}), functional differential equations (Faria and Magalh$\tilde{a}$es \cite{FariaTeresaMagalhãesLuisT-JDE-1995Hopf,FariaTeresaMagalhãesLuisT-JDE-1995BT}), and so on. More recently, a normal form theory has been well-developed by Liu et al. \cite{LiuZhihuaMagalPierreRuanShigui-JDE-2014} for the non-densely defined abstract Cauchy problems.

In document \cite{LiJunZhaoYulinZhuHuaiping-JMAA-2015}, the authors consider the following system
\begin{equation}\label{system0}
   \left\{
     \begin{array}{rl}
      \frac{dS(t)}{dt}&=N-\mu S(t)-\frac{\alpha S(t)I(t)^{2}}{1+\beta I(t)^{2}}  +\delta_{0} I(t), \\
       \frac{dI(t)}{dt}&=\frac{\alpha S(t)I(t)^{2}}{1+\beta I(t)^{2}} - (\mu+d+ \delta_{0})I(t) ,\\
     \end{array}
   \right.
\end{equation}
where $S(t)$ and $I(t)$ denote the numbers of susceptible and infectious individuals at time $t$, respectively. $N>0$ is the constant recruitment rate for susceptible individuals. $\mu$ $(0<\mu<1)$ is the natural death rate of the population. $\frac{\alpha I(t)^{2}}{1+\beta I(t)^{2}}$ is the effective contact rate that is proposed by Liu et al. \cite{LiuLevinIwasa-JMB-1986}. The authors denote the disease-induced death rate by $d$ $(0<d<1)$. The infective individuals recover with the rate $\delta_{0}$ $(0<\delta_{0}<1)$.  For simplicity of presentation, by rescaling the
parameters and phase variables, the system (\ref{system0}) can be rewritten
as
\begin{equation}\label{system01}
  \left\{
     \begin{array}{ccl}
       \frac{dS(t)}{dt} & = & \Lambda-\mu S(t)-\frac{S(t)I(t)^2}{1+\beta I(t)^2}+\eta I(t), \\
       \frac{dI(t)}{dt} & = & \frac{S(t)I(t)^2}{1+\beta I(t)^2}- I(t). \\
     \end{array}
   \right.
\end{equation}

Motivated by references \cite{ChuLiuMagalRuan-JDDE-2016,LiuZhihuaMagalPierreRuanShigui-JDE-2014},
we reconsider the SIS model (\ref{system01}) as an infection-age structure with infective individuals.
Now we concentrate our attention on the SIS model (\ref{system01}) in which $I(t)$
is structured by the infection-age. Let $a$ be the infection-age variable. $i(t,a)$ is the distribution function of $I$ over infection-age $a$ at time $t$. Then the infected individuals $I(t)$ at time $t$ equals to $\int_{0}^{+\infty}{i(t,a)da}$. Here, $\frac{1}{1+\int_{0}^{+\infty}{\beta(a)i(t,a)da}}$ describes the inhibition effect from the behavioral change of the susceptible individuals when the number of infectious individuals increase, and $\beta(a)$ is a influence function related to infection-age $a$ and satisfies the following assumption \ref{assumption1}. Inspired by the model (\ref{system01}), we consider the nonlinear incidence
\begin{equation*}
\frac{S(t)\int_{0}^{+\infty}{i(t,a)da}}{1+\int_{0}^{+\infty}{\beta(a)i(t,a)da}}.
\end{equation*}
\begin{assumption}\label{assumption1}
Assume that
\begin{equation*}
  \beta(a):=\left\{
               \begin{array}{cl}
                 \beta^{*}, & \mbox{   if   } a\geq \tau, \\
                 0, & \mbox{   if   } a\in (0,\tau), \\
               \end{array}
             \right.
\end{equation*}
where $\tau>0$ and $\beta^{*}>0$. Additionally, it will be convenient for the computation to assume that $\int_{0}^{+\infty}{\beta(a)e^{-a}da}=1$, where $e^{-a}$ denotes the probability of being affected.
\end{assumption}
\noindent
Correspondingly, we derive a new SIS infection-age structured epidemic model
\begin{equation}\label{system}
\left\{
\begin{array}{l}
\frac{dS(t)}{dt}=\Lambda-\mu S(t)-\frac{S(t)\int_{0}^{+\infty}i(t,a)da}{1+\int_{0}^{+\infty}\beta(a)i(t,a)da} +\eta\int_{0}^{+\infty}{i(t,a)da},\\
\frac{\partial i(t,a)}{\partial t}+\frac{\partial i(t,a)}{\partial a} = -i(t,a),\\
  i(t,0)=\frac{S(t)\int_{0}^{+\infty}i(t,a)da}{1+\int_{0}^{+\infty}\beta(a)i(t,a)da}, t>0,\\
  S(0)=S_{0}\geq0, i(0,\cdot)=i_{0}\in L_{+} ^{1}((0, + \infty ),\mathbb{R}),
\end{array}
\right.
\end{equation}
where $\Lambda$ is the constant recruitment rate for susceptible individuals, and $\mu(0<\mu<1)$ is the rate of natural death. The infective individuals recover with the rate $\eta(0<\eta<1)$.
To the best of our knowledge,
the age structure model can be considered as abstract Cauchy problems with non-dense domain by hardly applying the theory of integrated semigroup, combined with Hopf bifurcation theory and normal form theory.
In this paper, authors attempt to investigate the model (\ref{system}) by means of the methods mentioned above.
Furthermore, the explicit formulae are given which determine the direction of the Hopf bifurcation and the stability of the bifurcating periodic solutions. Numerical simulations are also presented to support our theoretical results.

The paper is organized as follows. In Section 2, we reformulate system (\ref{system}) as an abstract non-densely defined Cauchy problem and study the equilibrium, linearized equation and characteristic equation. The existence of Hopf bifurcation is proved in Section 3. In Section 4, we apply the normal form theory to system (\ref{system}), and then the explicit formulae are given to determine the direction of the Hopf bifurcation and the stability of the bifurcating periodic solutions. Some numerical simulations and conclusions are presented In Section 5.

\section{Preliminaries}
\noindent
\subsection{Rescaling time and age}

\noindent

In order to use the parameter $\tau$ as a bifurcation parameter (i.e. in order to obtain a smooth dependency of the system (\ref{system}) with respect to $\tau$), we first normalize $\tau$ in (\ref{system}) by the time-scaling and age-scaling
\begin{equation*}
  \hat{a}=\frac{a}{\tau}  \mbox{  and   }  \hat{t}=\frac{t}{\tau},
\end{equation*}
and consider the following distribution
\begin{equation*}
  \hat{S}(\hat{t})=S(\tau\hat{t})  \mbox{   and   }   \hat{i}(\hat{t},\hat{a})=\tau i(\tau\hat{t},\tau\hat{a}).
\end{equation*}
By dropping the hat notation we obtain, after the change of variables, the new system
\begin{equation}\label{newsystem}
\left\{
\begin{array}{l}
\frac{dS(t)}{dt}=\tau\left[\Lambda-\mu S(t)-\frac{S(t)\int_{0}^{+\infty}i(t,a)da}{1+\int_{0}^{+\infty}\beta(a)i(t,a)da} +\eta\int_{0}^{+\infty}{i(t,a)da}\right],\\
\frac{\partial i(t,a)}{\partial t}+\frac{\partial i(t,a)}{\partial a} = -\tau i(t,a),\\
  i(t,0)=\tau\frac{S(t)\int_{0}^{+\infty}i(t,a)da}{1+\int_{0}^{+\infty}\beta(a)i(t,a)da}, t>0,\\
  S(0)=S_{0}\geq0, i(0,\cdot)=i_{0}\in L_{+} ^{1}((0, + \infty ),\mathbb{R}),
\end{array}
\right.
\end{equation}
with the new function $\beta(a)$ defined by
\begin{equation*}
  \beta(a)=\left\{
             \begin{array}{cc}
               \beta^{*}, & \mbox{  if   } a \geq 1,\\
               0, & \mbox{  otherwise}, \\
             \end{array}
           \right.
\end{equation*}
and
\begin{equation*}
  \int_{\tau}^{+\infty}{\beta^{*}e^{-a}da}=1\Leftrightarrow \beta^{*}=e^{\tau},
\end{equation*}
where $\tau\geq0$, $\beta^{*}>0$.
\noindent
In system (\ref{newsystem}), by setting $S(t):=\int_{0}^{+\infty}{\rho(t,a)da}$. We can rewrite the ordinary differential equation in (\ref{newsystem}) as an age-structured model
\begin{equation*}
  \left\{
     \begin{array}{ll}
       \frac{\partial \rho(t,a)}{\partial t}+\frac{\partial \rho(t,a)}{\partial a} =-\tau\mu \rho(t,a),\\
       \rho(t,0)= \tau G(i(t,a),\rho(t,a)), \\
       \rho(0,a)=\rho_{0}\in L^{1}((0,+\infty),\mathbb{R}), \\
     \end{array}
   \right.
\end{equation*}
where
\begin{equation*}
  G(i(t,a),\rho(t,a))= \Lambda-\frac{\int_{0}^{+\infty}{\rho(t,a)da}\int_{0}^{+\infty} i(t,a)da}{1+ \int_{0}^{+\infty}\beta(a) i(t,a)da} +\eta\int_{0}^{+\infty}{i(t,a)da}.
\end{equation*}
Accordingly, by setting $u(t,a)=\left(
                           \begin{array}{c}
                             i(t,a) \\
                             \rho(t,a) \\
                           \end{array}
                         \right)
$, we obtain the equivalent system of model (\ref{system})
\begin{equation}\label{systempartialwta}
\left\{
  \begin{array}{l}
    \frac{\partial u(t,a)}{\partial t}+\frac{\partial u(t,a)}{\partial a} =-\tau Q u(t,a), \\
    u(t,0)=\tau B(u(t,a)), \\
    u(0,\cdot)=u_{0}=\left(
                       \begin{array}{c}
                         i_{0} \\
                         \rho_{0} \\
                       \end{array}
                     \right)\in L^{1}((0,+\infty),\mathbb{R}^{2}),
     \\
  \end{array}
\right.
\end{equation}
where
\begin{equation*}
  \begin{array}{ccc}
    Q=\left(
        \begin{array}{cc}
          1 & 0 \\
          0 & \mu \\
        \end{array}
      \right)
     & \mbox{and} & B(u(t,a))=\left(
                                   \begin{array}{c}
                                     \frac{\int_{0}^{+\infty}{\rho(t,a)da}\int_{0}^{+\infty} i(t,a)da}{1+ \int_{0}^{+\infty}\beta(a) i(t,a)da} \\
                                     G(i(t,a),\rho(t,a)) \\
                                   \end{array}
                                 \right).\\
  \end{array}
\end{equation*}
\noindent
Following the results developed in Magal et al. \cite{MagalMcCluskeyWebb-AA-2010}, we consider the following Banach space
\begin{equation*}
  X={\mathbb{R}}^{2} \times L^{1}{((0,+\infty),{\mathbb{R}}^{2})}
\end{equation*}
with $\left \|
\left(
  \begin{array}{c}
    \alpha \\
    \psi\\
  \end{array}
\right)
     \right \|
     =\left \|\alpha\right\|_{{\mathbb{R}}^{2}}+\left\|\psi\right\|_{L^{1}{((0,+\infty),{\mathbb{R}}^{2})}}$.
\noindent
Define the linear operator $A_{\tau} : D(A_{\tau})\rightarrow X$ by
\begin{equation*}
  A_{\tau}\left(
     \begin{array}{c}
       0_{\mathbb{R}^{2}} \\
       \varphi \\
     \end{array}
   \right)
   =\left(
   \begin{array}{c}
     -\varphi(0) \\
     -\varphi'-\tau Q\varphi \\
   \end{array}
 \right)
\end{equation*}
with $D(A_{\tau})=\{0_{\mathbb{R}^{2}}\}\times W^{1,1}({(0,+\infty),{\mathbb{\mathbb{R}}}^{2}}) \subset X$, and the operator $H: \overline{D(A_{\tau})} \rightarrow X$ by
\begin{equation*}
  H\left(
  \left(
     \begin{array}{c}
       0_{\mathbb{R}^{2}} \\
       \varphi \\
     \end{array}
   \right)
   \right)
   =\left(
      \begin{array}{c}
        B(\varphi) \\
        0_{L^{1}} \\
      \end{array}
    \right).
\end{equation*}
The linear operator $A_{\tau}$ is non-densely defined owing to
\begin{equation*}
  X_{0}:=\overline{D(A_{\tau})}=\{0_{\mathbb{R}^{2}}\} \times L^{1}{((0,+\infty),{\mathbb{R}}^{2})}.
\end{equation*}
By denoting that
\begin{equation*}
  v(t)=\left(
         \begin{array}{c}
           0_{\mathbb{R}^{2}} \\
           u(t,\cdot) \\
         \end{array}
       \right),
\end{equation*}
we can rewrite system (\ref{systempartialwta}) as the following non-densely defined abstract Cauchy problem
\begin{equation}\label{nonddaCp}
  \left\{
    \begin{array}{l}
      \frac{dv(t)}{dt}  =  A_{\tau}v(t)+\tau H(v(t)), t\geq0, \\
      v(0)  =  \left(
                   \begin{array}{c}
                     0_{\mathbb{R}^{2}} \\
                     u_0 \\
                   \end{array}
                 \right)
                 \in \overline{D(A_{\tau})}. \\
    \end{array}
  \right.
\end{equation}
The global existence and uniqueness of solution of equation (\ref{nonddaCp}) follow from the results of Magal and Ruan \cite{MagalRuan-ADE-2009} and Magal\cite{Magal-EJDE-2001}.
\subsection{Equilibrium and linearized equation}

\subsubsection{Existence of positive equilibrium}

\noindent

If $\overline{v}(a)=\left(
                      \begin{array}{c}
                        0_{\mathbb{R}^{2}} \\
                        \overline{u}(a) \\
                      \end{array}
                    \right)
                    \in X_0
$ is a steady state of system (\ref{nonddaCp}), we have
\begin{equation*}
  \left(
     \begin{array}{c}
       0_{\mathbb{R}^{2}} \\
       \overline{u}(a) \\
     \end{array}
   \right)\in D(A_{\tau}) \mbox{  and  }
   A_{\tau}\left(
      \begin{array}{c}
        0_{\mathbb{R}^{2}} \\
        \overline{u}(a) \\
      \end{array}
    \right)+\tau H\left(\left(
               \begin{array}{c}
                 0_{\mathbb{R}^{2}} \\
                 \overline{u}(a) \\
               \end{array}
             \right)\right)=0,
\end{equation*}
which is equivalent to
\begin{equation*}
\left\{
  \begin{array}{l}
    -\overline{u}(0)+\tau B(\overline{u}(a))=0, \\
    -\overline{u}^{'}(a)-\tau Q\overline{u}(a)=0. \\
  \end{array}
\right.
\end{equation*}
Therefore, we obtain
\begin{equation}\label{overlinewa}
  \left.
    \begin{array}{ccccc}
      \overline{u}(a)  = \left(
                              \begin{array}{c}
                                \overline{i}(a) \\
                                 \overline{\rho}(a) \\
                              \end{array}
                            \right)
        = \left(
               \begin{array}{c}
                 \tau \frac{\overline{S}\int_{0}^{+\infty}\overline{i}(a)da}{1+ \int_{0}^{+\infty}\beta(a)\overline{i}(a)da}e^{-\tau a}\\
                  \tau \left(\Lambda-\frac{\overline{S}\int_{0}^{+\infty} \overline{i}(a)da}{1+ \int_{0}^{+\infty}\beta(a) \overline{i}(a)da} +\eta\int_{0}^{+\infty}{\overline{i}(a)da}\right)e^{-\mu\tau a} \\
               \end{array}
             \right)
        \\
    \end{array}
  \right.
\end{equation}
with $\overline{S}=\int_{0}^{+\infty}{\overline{\rho}(a)da}$.
\noindent
From (\ref{overlinewa}) we can get
\begin{equation}\label{overlineIa}
\overline{i}(a)=\frac{\tau\overline{ S}\int_{0}^{+\infty}\overline{i}(a)da}{1+\int_{0}^{+\infty}\beta(a)\overline{i}(a)da}e^{-\tau a}
\end{equation}
and
\begin{equation}\label{overlinerhoa}
 \overline{\rho}(a)=\tau \left(\Lambda-\frac{\overline{S}\int_{0}^{+\infty} \overline{i}(a)da}{1+ \int_{0}^{+\infty}\beta(a) \overline{i}(a)da} +\eta\int_{0}^{+\infty}{\overline{i}(a)da}\right)e^{-\mu\tau a}.
\end{equation}
On the basis of (\ref{overlineIa}), we can yield
\begin{equation}\label{one}
  1=\frac{\overline{S}}{1+\int_{0}^{+\infty}{\beta(a)\overline{i}(a)da}}
\end{equation}
and
\begin{equation}\label{intdeltaaoverlineIa}
  \int_{0}^{+\infty}{\beta(a)\overline{i}(a)da}=\frac{\tau\overline{S}\int_{0}^{+\infty}\overline{i}(a)da}{1+ \int_{0}^{+\infty}\beta(a)\overline{i}(a)da}\int_{0}^{+\infty}{\beta(a)e^{-\tau a}da}=\int_{0}^{+\infty}{\overline{i}(a)da}.
\end{equation}
Using the second equation of (\ref{overlinewa}), we obtain
\begin{equation}\label{odeLambda}
  \Lambda-\mu \overline{S}-\frac{\overline{S}\int_{0}^{+\infty} \overline{i}(a)da}{1+\int_{0}^{+\infty}\beta(a)\overline{i}(a)da} +\eta\int_{0}^{+\infty}{\overline{i}(a)da}=0.
\end{equation}
Substituting (\ref{one}) and (\ref{intdeltaaoverlineIa}) into (\ref{odeLambda}), we can rewrite (\ref{odeLambda}) as
\begin{equation*}
  \Lambda-\mu\overline{S}+(\eta-1)\int_{0}^{+\infty}{\overline{i}(a)da}=0.
\end{equation*}
Solving the above equation, we obtain
\begin{equation}\label{overlineS}
  \overline{S}=\frac{\Lambda}{\mu}+\frac{\eta-1}{\mu}\int_{0}^{+\infty}{\overline{i}(a)da}.
\end{equation}
In addition, it follows from (\ref{one}) that
\begin{equation}\label{subS}
 1+\int_{0}^{+\infty}{\beta(a)\overline{i}(a)da}=\overline{S}.
\end{equation}
Combining (\ref{overlineS}) with (\ref{subS}), we have
\begin{equation*}
\int_{0}^{+\infty}{\overline{i}(a)da}=\frac{\Lambda-\mu}{1+\mu-\eta}.
\end{equation*}
Therefore, in accordance with (\ref{overlineIa}) and (\ref{overlinerhoa}), we obtain the following lemma.

\begin{lemma}
System (\ref{nonddaCp}) has always the  equilibrium
  \begin{equation*}
  \overline{v}_{0}(a)=\left(
                        \begin{array}{c}
                          0_{\mathbb{R}^{2}} \\
                          \left(
                            \begin{array}{c}
                              0_{L^{1}} \\
                              \Lambda\tau e^{-\mu\tau a} \\
                            \end{array}
                          \right)
                           \\
                        \end{array}
                      \right).
\end{equation*}
Furthermore, there exists a unique positive equilibrium of system (\ref{nonddaCp})
\begin{equation*}
\overline{v}_{\tau}=\left(
                  \begin{array}{c}
                    0_{\mathbb{R}^{2}} \\
                    \overline{u}_{\tau} \\
                  \end{array}
                \right)=\left(
  \begin{array}{c}
    0_{\mathbb{R}^{2}} \\
    \left(
      \begin{array}{c}
        \frac{\Lambda-\mu}{1+\mu-\eta}\tau e^{-\tau a} \\
        \frac{\mu(1+\Lambda-\eta)}{1+\mu-\eta}\tau e^{-\mu\tau a} \\
      \end{array}
    \right)
  \end{array}
\right)
\end{equation*}
if and only if
\begin{equation*}
  \Lambda-\mu>0.
\end{equation*}
\end{lemma}

\subsubsection{Linearized equation}

%\noindent

In this section, we need to get the linearized equation around the positive equilibrium $\overline{v}_{\tau}$. By using the following change of variable
\begin{equation*}
  w(t):= v(t)-\overline{v}_{\tau},
\end{equation*}
we obtain
\begin{equation}\label{system4}
  \left\{
    \begin{array}{cll}
      \frac{dw(t)}{dt} & = & A_{\tau}w(t)+\tau H(w(t)+\overline{v}_{\tau})-\tau H(\overline{v}_{\tau}), t\geq0, \\
      w(0) & = & \left(
                   \begin{array}{c}
                     0_{\mathbb{R}^{2}} \\
                     u_0-\overline{u}_{\tau} \\
                   \end{array}
                 \right)
                 =: w_0\in \overline{D(A_{\tau})}.
       \\
    \end{array}
  \right.
\end{equation}
Therefore the linearized equation (\ref{system4}) around the equilibrium $0$ is given by
\begin{equation}\label{systemlinear}
  \begin{array}{cc}
    \frac{dw(t)}{dt}=A_{\tau}w(t)+\tau DH(\overline{v}_{\tau})w(t) & \mbox{   for   } t\geq 0, w(t)\in X_{0}, \\
  \end{array}
\end{equation}
where
\begin{equation*}
    \begin{array}{cc}
      \tau DH(\overline{v}_{\tau})\left(
                        \begin{array}{c}
                          0_{\mathbb{R}^{2}} \\
                          \varphi \\
                        \end{array}
                      \right)=\left(
                                \begin{array}{c}
                                 \tau DB(\overline{u}_{\tau})(\varphi) \\
                                  0_{L^{1}} \\
                                \end{array}
                              \right)
       & \mbox{for all}\left(
                         \begin{array}{c}
                           0_{\mathbb{R}^{2}} \\
                           \varphi \\
                         \end{array}
                       \right)\in D(A_{\tau})\\
    \end{array}
\end{equation*}
with
\begin{equation*}
     \begin{array}{ccl}
       DB(\overline{u}_{\tau})(\varphi) & = & \left(
                              \begin{array}{cc}
                                1 & \frac{\int_{0}^{+\infty}{\overline{i}(a)da}}{1+\int_{0}^{+\infty}{\beta(a)\overline{i}(a)da}} \\
                                \eta-1 & -\frac{\int_{0}^{+\infty}{\overline{i}(a)da}}{1+\int_{0}^{+\infty}{\beta(a)\overline{i}(a)da}} \\
                              \end{array}
                            \right)\int_{0}^{+\infty}{\varphi(a)da} \\
        &  & +\left(
                              \begin{array}{cc}
                                -\frac{\int_{0}^{+\infty}{\overline{i}(a)da}}{1+\int_{0}^{+\infty}{\beta(a)\overline{i}(a)da}} & 0_{\mathbb{R}} \\
                                \frac{\int_{0}^{+\infty}{\overline{i}(a)da}}{1+\int_{0}^{+\infty}{\beta(a)\overline{i}(a)da}} & 0_{\mathbb{R}} \\
                              \end{array}
                            \right)
        \int_{0}^{+\infty}{\beta(a)\varphi(a)da}. \\
     \end{array}
\end{equation*}
Then we can rewrite system (\ref{system4}) as
\begin{equation}\label{fracdytdt}
    \begin{array}{cc}
      \frac{dw(t)}{dt}=B_{\tau}w(t)+\mathcal{H}(w(t)) & \mbox{   for  } t\geq0, \\
    \end{array}
\end{equation}
where
\begin{equation*}
  B_{\tau}:=A_{\tau}+\tau DH(\overline{v}_{\tau})
\end{equation*}
is a linear operator and
\begin{equation*}
  \mathcal{H}(w(t))=\tau H(w(t)+\overline{v}_{\tau})-\tau H(\overline{v}_{\tau})-\tau DH(\overline{v}_{\tau})w(t)
\end{equation*}
satisfying $\mathcal{H}(0)=0$ and $D\mathcal{H}(0)=0$.

\subsection{Characteristic equation}

\noindent

Let $\Omega := \{\lambda \in \mathbb{C} : Re(\lambda)>-\mu\tau\}$. Applying the results of Liu et al.\cite{LiuMagalRuan-ZAMP-2011}, we obtain the following result.
\begin{lemma}
For $\lambda\in \Omega$, $\lambda\in \rho(A_{\tau})$ and
\begin{equation*}
  (\lambda I-A_{\tau})^{-1}\left(
                               \begin{array}{c}
                                 \delta \\
                                 \psi\\
                               \end{array}
                             \right)
                             =\left(
                                \begin{array}{c}
                                  0_{\mathbb{R}^{2}} \\
                                  \varphi \\
                                \end{array}
                              \right)
                              \Leftrightarrow
                              \varphi(a)=e^{-\int_{0}^{a}{(\lambda I+\tau Q)dl}}\delta+\int_{0}^{a}{e^{-\int_{s}^{a}{(\lambda I+\tau Q)dl}}\psi(s)}ds
\end{equation*}
with $\left(
        \begin{array}{c}
          \delta \\
          \psi \\
        \end{array}
      \right)
      \in X
$ and $\left(
         \begin{array}{c}
           0_{\mathbb{R}^{2}} \\
           \varphi \\
         \end{array}
       \right)
       \in D(A_{\tau})
$. Moreover, $A_{\tau}$ is a Hille-Yosida operator and
\begin{equation}\label{Hille-Yosida}
  \left\|(\lambda I-A_{\tau})^{-n}\right\|\leq\frac{1}{(Re(\lambda)+\mu\tau)^{n}}, \forall \lambda\in\Omega,\forall n\geq 1.
\end{equation}
\end{lemma}
\noindent
Let $A_0$ be the part of $A_{\tau}$ in $\overline{D(A_{\tau})}$, that is, $A_0 := D(A_0)\subset X \rightarrow X$. For $\left(
                                                                                                    \begin{array}{c}
                                                                                                      0_{\mathbb{R}^{2}} \\
                                                                                                      \varphi \\
                                                                                                    \end{array}
                                                                                                  \right)
                                                                                                  \in D(A_0)
$, we get
\begin{equation*}
  A_0\left(
       \begin{array}{c}
         0_{\mathbb{R}^{2}} \\
         \varphi \\
       \end{array}
     \right)
     =\left(
        \begin{array}{c}
          0_{\mathbb{R}^{2}} \\
          \hat{A_0}(\varphi) \\
        \end{array}
      \right),
\end{equation*}
where $\hat{A_0}(\varphi)=-\varphi '-\tau Q\varphi$ with $D(\hat{A_0})=\{\varphi \in W^{1,1}((0,+\infty),{\mathbb{R}}^{2}): \varphi(0)=0\}$.
\noindent
Note that $\tau DH(\overline{v}_{\tau}):D(A_{\tau}) \subset X \rightarrow X$ is a compact bounded linear operator. From (\ref{Hille-Yosida}) we have
\begin{equation*}
  \left\| T_{A_0}(t) \right\| \leq e^{-\mu\tau t} \mbox{  for  } t \geq 0.
\end{equation*}
Thus, we have
\begin{equation*}
  \omega_{0,ess}(A_0)\leq\omega_0((B_{\tau})_{0})\leq -\mu\tau.
\end{equation*}
By applying the perturbation results in Ducrot et al. \cite{DucrotLiuMagal-JMAA-2008}, we obtain
\begin{equation*}
  \omega_{0,ess}((A_{\tau}+\tau DH(\overline{v}_{\tau}))_{0})\leq-\mu\tau<0.
\end{equation*}
Hence we have the following proposition.
\begin{lemma}
The linear operator $A_{\tau}$ is a Hille-Yosida operator, and its parts $A_{0}$ in
$\overline{D(A_{\tau})}$ satisfies
\begin{equation*}
  \omega_{0,ess}(A_{0})<0.
\end{equation*}
\end{lemma}
\noindent
For convenience, we set $C:=\tau DH(\overline{v}_{\tau})$. Let $\lambda\in \Omega$. Since $(\lambda I-A_{\tau})$ is invertible, and
\begin{equation}\label{invertible}
  \begin{array}{ccl}
    (\lambda I-B_{\tau})^{-1} & = & (\lambda I-(A_{\tau}+C))^{-1} \\
     & = & (\lambda I-A_{\tau})^{-1}(I-C(\lambda I-A_{\tau})^{-1})^{-1}, \\
  \end{array}
\end{equation}
it follows that $\lambda I-(A_{\tau}+\tau DH(\overline{v}_{\tau}))$ is invertible if and only if $I-C(\lambda I-A_{\tau})^{-1}$ is invertible. Moreover, when $I-C(\lambda I-A_{\tau})^{-1}$ is invertible, we have
\begin{equation*}
  (I-C(\lambda I-A_{\tau})^{-1})\left(
     \begin{array}{c}
       \delta \\
       \varphi \\
     \end{array}
   \right)
   =\left(
      \begin{array}{c}
        \gamma \\
        \psi \\
      \end{array}
    \right).
\end{equation*}
It follows that
\begin{equation*}
  \left(
    \begin{array}{l}
      \delta \\
      \varphi \\
    \end{array}
  \right)
  -\tau DH(\overline{v}_{\tau})(\lambda I-A_{\tau})^{-1}\left(
                                                \begin{array}{c}
                                                  \delta \\
                                                  \varphi \\
                                                \end{array}
                                              \right)
                                              =\left(
                                                 \begin{array}{c}
                                                   \gamma \\
                                                   \psi \\
                                                 \end{array}
                                               \right).
\end{equation*}
Then we obtain
\begin{equation*}
 \left\{
    \begin{array}{l}
      \delta-\tau DB(\overline{u}_{\tau})\left(e^{-\int_{0}^{a}{(\lambda I+\tau Q)dl}}\delta+\int_{0}^{a}{e^{-\int_{s}^{a}{(\lambda I+\tau Q)dl}}\varphi(s)}ds\right)=\gamma, \\
      \varphi=\psi, \\
    \end{array}
  \right.
\end{equation*}
i.e.,
\begin{equation*}
  \left\{
    \begin{array}{l}
 \delta-\tau DB(\overline{u}_{\tau})\left(e^{  -\int_{0}^{a}{(\lambda I+\tau Q)dl}}\delta\right)=\gamma+\tau DB(\overline{u}_{\tau})\left(\int_{0}^{a}{e^{-\int_{s}^{a}{(\lambda I+\tau Q)dl}}\varphi(s)}ds\right), \\
      \varphi=\psi. \\
    \end{array}
  \right.
\end{equation*}
From the formula of $DB(\overline{w})$ we know that
\begin{equation*}
  \left\{
    \begin{array}{l}
      \Delta(\lambda)\delta=\gamma+K(\lambda,\psi), \\
      \varphi=\psi, \\
    \end{array}
  \right.
\end{equation*}
where
\begin{equation}\label{Deltalambda}
   \begin{array}{ccl}
     \Delta(\lambda) & = & I-\left(
                              \begin{array}{cc}
                                1 & \frac{\int_{0}^{+\infty}{\overline{i}(a)da}}{1+\int_{0}^{+\infty}{\beta(a)\overline{i}(a)da}} \\
                                \eta-1 & -\frac{\int_{0}^{+\infty}{\overline{i}(a)da}}{1+\int_{0}^{+\infty}{\beta(a)\overline{i}(a)da}} \\
                              \end{array}
                            \right)\tau
                             \int_{0}^{+\infty}{e^{-\int_{0}^{a}{(\lambda I+\tau Q)dl}}da} \\
      &  &  -\left(
                              \begin{array}{cc}
                                -\frac{\int_{0}^{+\infty}{\overline{i}(a)da}}{1+\int_{0}^{+\infty}{\beta(a)\overline{i}(a)da}} & 0_{\mathbb{R}} \\
                                \frac{\int_{0}^{+\infty}{\overline{i}(a)da}}{1+\int_{0}^{+\infty}{\beta(a)\overline{i}(a)da}} & 0_{\mathbb{R}} \\
                              \end{array}
                            \right)\tau
                             \int_{0}^{+\infty}{\beta(a)e^{-\int_{0}^{a}{(\lambda I+\tau Q)dl}}}da \\
   \end{array}
\end{equation}
and
\begin{equation}\label{Klambdapsi}
  K(\lambda,\psi)=\tau DB(\overline{u}_{\tau})\left(\int_{0}^{a}{e^{-\int_{s}^{a}{(\lambda I+\tau Q)dl}}\psi(s)}ds\right).
\end{equation}
Whenever $\Delta(\lambda)$ is invertible, we have
\begin{equation}\label{xi}
  \delta=(\Delta(\lambda))^{-1}(\gamma+K(\lambda,\psi)).
\end{equation}
From the above discussion and by using the proof of Lemma 3.5 in \cite{WangLiu-JMAA-2012}, we obtain the following lemma.
\begin{lemma}
The following results hold
\begin{itemize}
  \item [(i)] $\sigma(B_{\tau})\cap\Omega=\sigma_{p}(B_{\tau})\cap\Omega=\{\lambda\in\Omega: \det(\Delta(\lambda))=0\}$;
  \item [(ii)] If $\lambda\in\rho(B_{\tau})\cap\Omega$, we have the following formula for resolvent
  \begin{equation}\label{lambdalambdaI}
    (\lambda I -B_{\tau})^{-1}\left(
                                \begin{array}{c}
                                  \delta \\
                                  \varphi \\
                                \end{array}
                              \right)
                              =\left(
                                 \begin{array}{c}
                                   0_{\mathbb{R}^{2}} \\
                                   \psi \\
                                 \end{array}
                               \right),
  \end{equation}
  where
  \begin{equation*}
    \psi(a)= e^{-\int_{0}^{a}{(\lambda I+\tau Q)dl}}(\Delta(\lambda))^{-1}\left[\delta+K(\lambda,\varphi)\right]+\int_{0}^{a}
    {e^{-\int_{s}^{a}{(\lambda I+\tau Q)dl}}}\varphi(s)ds
  \end{equation*}
  with $\Delta(\lambda)$ and $K(\lambda,\varphi)$ defined in (\ref{Deltalambda}) and (\ref{Klambdapsi}).
\end{itemize}
\end{lemma}
\noindent
Under Assumption \ref{assumption1}, we have
\begin{equation}\label{intea}
  \int_{0}^{+\infty}{e^{-\int_{0}^{a}({\lambda I+\tau Q})dl}}da=
  \left(
    \begin{array}{cc}
      \frac{1}{\lambda+\tau} & 0 \\
      0 & \frac{1}{\lambda+\mu\tau} \\
    \end{array}
  \right)
\end{equation}
and
\begin{equation}\label{intbetaea}
  \int_{0}^{+\infty}{\beta(a)e^{-\int_{0}^{a}({\lambda I+\tau Q})dl}}da=
  \left(
    \begin{array}{cc}
      \frac{\beta^{*}e^{-(\lambda+\tau)}}{\lambda+\tau} & 0 \\
      0 & \frac{\beta^{*}e^{-(\lambda+\mu\tau)}}{\lambda+\mu\tau} \\
    \end{array}
  \right).
\end{equation}
It follows from (\ref{Deltalambda}), (\ref{intea}) and (\ref{intbetaea}) that the characteristic equation at any feasible steady state is
\begin{equation}\label{characteristicequation}
     \begin{array}{ccl}
       \det(\Delta(\lambda)) & = & \left|
  \begin{array}{ll}
                  1-\frac{\tau}{\lambda+\tau}+\frac{\xi}{\xi+1}\frac{\tau e^{-\lambda}}{\lambda+\tau}& -\frac{\xi}{\xi+1}\frac{\tau}{\lambda+\mu\tau} \\
                  -\frac{\tau(\eta-1)}{\lambda+\tau}-\frac{\xi}{\xi+1}\frac{\tau e^{-\lambda}}{\lambda+\tau} & 1+\frac{\xi}{\xi+1}\frac{\tau}{\lambda+\mu\tau}\\
                 \end{array}
  \right| \\
        & = & \frac{\lambda^{2}+\tau B\lambda+\tau^{2} C+\tau D\lambda e^{-\lambda}+\tau^{2} Ee^{-\lambda}}{(\lambda+\tau)(\lambda+\mu\tau)} \\
        & \triangleq & \frac{\tilde{f}(\lambda)}{\tilde{g}(\lambda)}=0, \\
     \end{array}
\end{equation}
where $\xi=\int_{0}^{+\infty}{\overline{i}(a)da}$, $B=\frac{\mu(\xi+1)+\xi}{\xi+1}$, $C=\frac{\xi(1-\eta)}{\xi+1}$, $D=\frac{\xi}{\xi+1}$, $E=\frac{\mu\xi}{\xi+1}$, $\tilde{f}(\lambda)=\lambda^{2}+\tau B\lambda+\tau^{2} C+\tau D\lambda e^{-\lambda}+\tau^{2} Ee^{-\lambda}$ and $\tilde{g}(\lambda)=(\lambda+\tau)(\lambda+\mu\tau)$.
Let
\begin{equation*}
  \lambda=\tau\zeta.
\end{equation*}
Then we get
\begin{equation}\label{characteristicequationg}
  \tilde{f}(\lambda)=\tilde{f}(\tau\zeta):=\tau^{2}g(\zeta)=\tau^{2}(\zeta^{2}+ B\zeta+C+ D\zeta e^{-\tau\zeta}+ Ee^{-\tau\zeta}).
\end{equation}
It is easy to see that
\begin{equation*}
  \{\lambda\in\Omega:\det(\Delta(\lambda))=0\}=\{\lambda=\tau\zeta\in\Omega:g(\zeta)=0\}.
\end{equation*}

\section{Existence of Hopf bifurcation}

\noindent

In this section, we study the existence of Hopf bifurcation by applying the Hopf bifurcation theory to the Cauchy problem (\ref{nonddaCp}).
From (\ref{characteristicequationg}), we have
\begin{equation}\label{characteristicequation3}
 g(\zeta)=\zeta^2+B\zeta+C+D\zeta e^{-\tau\zeta}+E e^{-\tau\zeta},
\end{equation}
where
\begin{equation}\label{BCDE}
  \begin{array}{ccccc}
    D=\frac{\Lambda-\mu}{1+\Lambda-\eta}, & B=\frac{\mu(\Lambda-\eta)+\Lambda}{\Lambda-\mu}D, & C=(1-\eta)D & \mbox{ and } & E=\mu D. \\
  \end{array}
\end{equation}
It is obvious that $\zeta =0$ is not a eigenvalue of (\ref%
{characteristicequation3}). Let $\zeta=i\omega (\omega>0)$ be a purely imaginary root of $g(\zeta)=0$. Then we have
\begin{equation*}
  -\omega^{2}+iB\omega+C+iD\omega e^{-i\omega\tau}+Ee^{-i\omega\tau}=0.
\end{equation*}
Separating the real part and the imaginary part in the above equation, we yield
\begin{equation}\label{realimaginary}
  \left\{
     \begin{array}{l}
       \omega^{2}-C=D\omega\sin(\omega\tau)+E\cos(\omega\tau), \\
       -B\omega=D\omega\cos(\omega\tau)-E\sin(\omega\tau). \\
     \end{array}
   \right.
\end{equation}
 Consequently, we obtain
\begin{equation*}
  (\omega^{2}-C)^{2}+(-B\omega)^{2}=(D\omega)^2+E^{2},
\end{equation*}
i.e.
\begin{equation}\label{omega34}
  \omega^{4}+(B^{2}-2C-D^{2})\omega^{2}+C^{2}-E^{2}=0.
\end{equation}
Set $\sigma=\omega^{2}$, (\ref{omega34}) becomes
\begin{equation}\label{siama2}
 \sigma^{2}+(B^{2}-2C-D^{2})\sigma+C^{2}-E^{2}=0.
\end{equation}
Let $\sigma_{1}$ and $\sigma_{2}$ denote two roots of (\ref{siama2}), then we find
\begin{equation}\label{sigma1sigma2}
  \sigma_{1}+\sigma_{2}=-(B^{2}-2C-D^{2}), \sigma_{1}\sigma_{2}=C^{2}-E^{2}.
\end{equation}
We observe that when $\Lambda-\mu>0$, we get $C+E=\frac{(\Lambda-\mu)(1+\mu-\eta)}{1+\Lambda-\eta}>0$. accordingly when $C-E<0(i.e., \eta+\mu>1)$, it is apparent from (\ref{sigma1sigma2}) that (\ref{siama2}) has only one positive real root $\sigma_{0}$. Therefore (\ref{omega34}) has only one positive real root $\omega_{0}=\sqrt{\sigma_{0}}$. From (\ref{realimaginary}), we can yield that $g(\zeta)=0$ with $\tau=\tau_{k}$, $k=0,1,2,\cdots$ has a pair of purely imaginary roots $\pm i\omega$, where
\begin{equation*}
  \omega_{0}^{2}=\frac{-(B^{2}-2C-D^{2})+\sqrt{(B^{2}-2C-D^{2})^{2}-4(C^{2}-E^{2})}}{2}
\end{equation*}
and
\begin{equation}\label{tauk}
  \tau_{k}=\left\{
             \begin{array}{l}
               \frac{1}{\omega_{0}}\left(\arccos\frac{(E-BD)\omega_{0}^{2}-CE}{D^{2}\omega_{0}^2+E^{2}}+2k\pi\right),
               \mbox{  if  } \frac{\omega_{0}(D\omega_{0}^{2}+BE-CD)}{D^{2}\omega_{0}^2+E^{2}}\geq 0,\\
                \frac{1}{\omega_{0}}\left(2\pi-\arccos\frac{(E-BD)\omega_{0}^{2}-CE}{D^{2}\omega_{0}^2+E^{2}}+2k\pi\right),
                \mbox{  if  } \frac{\omega_{0}(D\omega_{0}^{2}+BE-CD)}{D^{2}\omega_{0}^2+E^{2}}< 0,\\
             \end{array}
           \right.
\end{equation}
for $k=0,1,2,\cdots.$
\begin{assumption}\label{assumption2}
  Assume that $\Lambda-\mu>0$, $C-E<0$, $2E-BD<0$, $2C-B^{2}<0$, and $E(B^{2}-2C)-BCD<0$, where $B$, $C$, $D$ and $E$ are given by (\ref{BCDE}).
\end{assumption}
\begin{lemma}
Let Assumption \ref{assumption1} and \ref{assumption2} hold, then
\begin{equation*}
  \frac{\mbox{d}g(\zeta)}{\mbox{d}\zeta}\Big|_{\zeta=i\omega_{0}}\neq0.
\end{equation*}
Hence $\zeta=i\omega_{0}$ is a simple root of (\ref{characteristicequation3}).
\end{lemma}
\begin{proof}
According to (\ref{characteristicequation3}), we obtain
\begin{equation*}
 \frac{\mbox{d}g(\zeta)}{\mbox{d}\zeta}\Big|_{\zeta=i\omega_{0}}=\left(2\zeta+B+(D-\tau D\zeta-\tau E)e^{-\tau\zeta}\right)\Big|_{\zeta=i\omega_{0}}.
\end{equation*}
Therefore, we find
\begin{equation*}
  \frac{\mbox{d}g(\zeta)}{\mbox{d}\zeta}\Big|_{\zeta=i\omega_{0}}=0\Leftrightarrow\left\{
                                                                        \begin{array}{l}
                                                                          B=\tau_{k} D \omega_{0}\sin(\omega_{0}\tau_{k})-(D-\tau_{k} E)\cos(\omega_{0}\tau_{k}), \\
                                                                          2\omega_{0}=\tau_{k} D \omega_{0}\cos(\omega_{0}\tau_{k})+(D-\tau_{k} E)\sin(\omega_{0}\tau_{k}). \\
                                                                        \end{array}
                                                                      \right.
\end{equation*}
Thus, we have
\begin{equation*}
  \tan(\omega_{0}\tau_{k})=\frac{\omega_{0}[(BD-2E)\tau_{k}+2D]}{(2D\omega_{0}^{2}+BE)\tau_{k}-BD}.
\end{equation*}
It follows from (\ref{realimaginary}) that
\begin{equation*}
  \tan(\omega_{0}\tau_{k})=-\frac{\omega_{0}(D\omega_{0}^{2}+BE-CD)}{\omega_{0}^{2}(BD-E)+CE}.
\end{equation*}
Consequently, we can yield
\begin{equation*}
  -2\tau_{k} D^{2}\omega_{0}^{4}+q_{1}\omega_{0}^{2}+q_{2}=0,
\end{equation*}
where
\begin{equation*}
       q_{1}=[D^{2}(2C-B^{2})-2E^{2}]\tau_{k}+D(2E-BD) \\
\end{equation*}
and
\begin{equation*}
q_{2}=E^{2}(2C-B^{2})\tau_{k}+D[E(B^{2}-2C)-BCD]. \\
\end{equation*}
Note that $\tau_{k}>0$, $\Lambda-\mu>0$, $D>0$, and $E>0$. If $2E-BD<0$, $2C-B^{2}<0$, and $E(B^{2}-2C)-BCD<0$, then
\begin{equation*}
  -2\tau_{k} D^{2}\omega_{0}^{4}+q_{1}\omega_{0}^{2}+q_{2}<0.
\end{equation*}
Accordingly, if assumption \ref{assumption1} and \ref{assumption2} hold, then $ \frac{\mathrm{d}g(\zeta)}{\mbox{d}\zeta}\Big|_{\zeta=i\omega_{0}}\neq0$.
\end{proof}
\begin{lemma}
  Let Assumption \ref{assumption1} and \ref{assumption2} hold. Denote the root $\zeta(\tau)=\alpha(\tau)+i\omega(\tau)$ of $g(\zeta)=0$ satisfying $\alpha(\tau_{k})=0$ and $\omega(\tau_{k})=\omega_{0}$, where $\tau_{k}$ is defined in (\ref{tauk}). Then
\begin{equation*}
  \alpha^{'}(\tau_{k})=\frac{\mbox{d}Re(\zeta)}{\mbox{d}\tau}\Big|_{\tau=\tau_{k}}>0.
\end{equation*}
\end{lemma}
\begin{proof}
For convenience, we study $\frac{d\tau}{d\zeta}$ instead of $\frac{d\zeta}{d\tau}$. From the expression of $g(\zeta)=0$, we have
\begin{equation*}
  \begin{array}{cll}
    \frac{\mbox{d}\tau}{\mbox{d}\zeta}\Big|_{\zeta=i\omega_{0}} & = & \frac{2\zeta+B+De^{-\tau\zeta}-\tau(\zeta D+E)e^{-\tau\zeta}}{\zeta(\zeta D+E)e^{-\tau\zeta}}\Big|_{\zeta=i\omega_{0}} \\
     & = & \left(-\frac{2\zeta+B}{\zeta(\zeta^{2}+B\zeta+C)}+\frac{D}{\zeta(\zeta D+E)}-\frac{\tau}{\zeta}\right)\Big|_{\zeta=i\omega_{0}} \\
     & = &-\frac{i2\omega_{0}+B}{i\omega_{0}(-\omega_{0}^{2}+iB\omega_{0}+C)}+\frac{D}{i\omega_{0}(i\omega_{0} D+E)}-\frac{\tau}{i\omega_{0}}\\
     & = &\frac{1}{\omega_{0}}\frac{i2\omega_{0}+B}{B\omega_{0}-i(C-\omega_{0}^{2})}+\frac{1}{\omega_{0}}\frac{-D}{D\omega_{0}-iE}-\frac{\tau}{i\omega_{0}}\\
     & = & \frac{1}{\omega_{0}}\frac{(i2\omega_{0}+B)[B\omega_{0}+i(C-\omega_{0}^{2})]}{(B\omega_{0})^{2}+(C-\omega_{0}^{2})^{2}}
     +\frac{1}{\omega_{0}}\frac{-D(D\omega_{0}+iE)}{(D\omega_{0})^{2}+E^{2}}+\frac{i\tau}{\omega_{0}}.\\
  \end{array}
\end{equation*}
Therefore, we have
\begin{equation*}
     \begin{array}{ccl}
       $\mbox{Re}$\left(\frac{\mbox{d}\tau}{\mbox{d}\zeta}\Big|_{\zeta=i\omega_{0}} \right) & = & \frac{2\omega_{0}^{2}+(B^{2}-2C)}{B^{2}\omega_{0}^{2}+(C-\omega_{0}^{2})^{2}}-\frac{D^{2}}{D^{2}\omega_{0}^{2}+E^{2}} \\
        & = & \frac{2\omega_{0}^{2}+B^{2}-2C-D^{2}}{D^{2}\omega_{0}^{2}+E^{2}}. \\
     \end{array}
\end{equation*}
Since
\begin{equation*}
  \omega_{0}^{2}=\frac{-(B^{2}-2C-D^{2})+\sqrt{(B^{2}-2C-D^{2})^{2}-4(C^{2}-E^{2})}}{2},
\end{equation*}
we can yield
\begin{equation*}
  \begin{array}{ccl}
    \mbox{sign}\left(\frac{\mbox{d}\rm{Re}(\zeta)}{\mbox{d}\tau}\Big|_{\tau=\tau_{k}}\right) & = & \mbox{sign}\left(\mbox{Re}\left(\frac{\mbox{d}\tau}{\mbox{d}\zeta}\Big|_{\zeta=i\omega_{0}}\right)\right) \\
     & = & \mbox{sign}\left(\frac{2\omega_{0}^{2}+B^{2}-2C-D^{2}}{\omega_{0}^{2}D^{2}+E^{2}}\right)>0. \\
  \end{array}
\end{equation*}
\end{proof}
\noindent
Summarizing the above results, we obtain the following theorem.
\begin{theorem}\label{HopfBifurcation}
Let Assumption \ref{assumption1} and \ref{assumption2} be satisfied. Then there exist $\tau_{k}>0 (k=0,1,2,\cdots)$($\tau_{k}$ is defined in (\ref{tauk})), such that when $\tau=\tau_{k}$, the SIS model (\ref{system}) undergoes a Hopf bifurcation at the equilibrium $(\overline{I}(a),\overline{S})$. In particular, a non-trivial periodic solution bifurcates from the equilibrium $(\overline{I}(a),\overline{S})$ when $\tau=\tau_{k}$.
\end{theorem}

\section{Direction and Stability of Hopf Bifurcation}

\noindent

In this section, we study the direction and stability of the Hopf bifurcation by applying the normal form theory developed in Liu et al. \cite{LiuZhihuaMagalPierreRuanShigui-JDE-2014} to the Cauchy problem (\ref{nonddaCp}).

\subsection{Conversion of Bifurcation Parameter}

\noindent

From now on, we include the parameter $\tau$ into the state variable. Consider the system
\begin{equation}\label{systmeincludetau}
  \left\{
    \begin{array}{l}
      \frac{d\tau(t)}{dt}=0, \\
      \frac{dv(t)}{dt}=A_{\tau(t)}v(t)+\tau(t) H(v(t)), \\
      \tau(0)=\tau_{0}\in \mathbb{R}, v(0)=v_{0}\in X_{0}. \\
    \end{array}
  \right.
\end{equation}
Making a change of variables
\begin{equation*}
  v(t)=\hat{v}(t)+\overline{v}_{\tau},
\end{equation*}
we obtain the system
\begin{equation*}
  \left\{
     \begin{array}{ccl}
       \frac{d\tau(t)}{dt} & = & 0, \\
       \frac{d\hat{v}(t)}{dt} & = & A_{\tau(t)}\hat{v}(t)+\tau(t) H\left(\hat{v}(t)+\overline{v}_{\tau}\right)-\tau(t) H(\bar{v}_{\tau}). \\
     \end{array}
   \right.
\end{equation*}
Setting
\begin{equation*}
  \tau=\hat{\tau}+\tau_{k},
\end{equation*}
we have
\begin{equation}\label{dhattaudhatx}
  \left\{
     \begin{array}{ccl}
       \frac{d\hat{\tau}(t)}{dt} & = & 0, \\
       \frac{d\hat{v}(t)}{dt} & = & \hat{A}\hat{v}(t)+\hat{H}(\hat{\tau},\hat{v}), \\
     \end{array}
   \right.
\end{equation}
where
\begin{equation*}
  \hat{A}:=A_{\hat{\tau}(t)+\tau_{k}}
\end{equation*}
and
\begin{equation*}
  \hat{H}(\hat{\tau},\hat{v}):=(\hat{\tau}+\tau_{k})[H(\hat{v}(t)+\overline{v}_{\hat{\tau}+\tau_{k}})-H(\overline{v}_{\hat{\tau}+\tau_{k}})].
\end{equation*}
Consequently, we can get
\begin{equation*}
     \begin{array}{ccl}
       \partial_{\hat{\tau}}\hat{H}(\hat{\tau},\hat{v})(\tilde{\tau}) & = & \tilde{\tau}\bigg\{ H(\hat{v}+\overline{v}_{\hat{\tau}+\tau_{k}})-H(\overline{v}_{\hat{\tau}+\tau_{k}}) \\
       &  & +(\hat{\tau}+\tau_{k})\bigg[DH(\hat{v}+\overline{v}_{\hat{\tau}+\tau_{k}})\left(\frac{d\overline{v}_{\hat{\tau}+\tau_{k}}}{d\hat{\tau}}\right) \\
        &  & -DH(\overline{v}_{\hat{\tau}+\tau_{k}})\left(\frac{d\overline{v}_{\hat{\tau}+\tau_{k}}}{d\hat{\tau}}\right)\bigg]\bigg\} \\
     \end{array}
\end{equation*}
and
\begin{equation*}
  \partial_{\hat{v}}\hat{H}(\hat{\tau},\hat{v})(\tilde{v})=(\hat{\tau}+\tau_{k})DH(\hat{v}+\overline{v}_{\hat{\tau}+\tau_{k}})(\tilde{v}).
\end{equation*}
Hence
\begin{equation*}
     \begin{array}{ccc}
       \partial_{\hat{\tau}}\hat{H}(0,0)(\tilde{\tau})=0 & \mbox{ and } & \partial_{\hat{v}}\hat{H}(0,0)(\tilde{v})=\tau_{k}DH(\overline{v}_{\tau_{k}})(\tilde{v}). \\
     \end{array}
\end{equation*}
Set
\begin{equation*}
     \begin{array}{ccc}
       \mathcal{X}=\mathbb{R}\times X &\mbox{and}  & \mathcal{X}_{0}=\mathbb{R}\times\overline{D(A)}. \\
     \end{array}
\end{equation*}
Consider the linear operator $\mathcal{A}: D(\mathcal{A})\subset\mathcal{X}\rightarrow\mathcal{X}$ defined by
\begin{equation*}
  \mathcal{A}\left(
               \begin{array}{c}
                 \hat{\tau} \\
                 \hat{v} \\
               \end{array}
             \right)=
             \left(
               \begin{array}{c}
                 0 \\
                 (A_{\tau_{k}}+\tau_{k}DH(\overline{v}_{\tau_{k}}))\hat{v} \\
               \end{array}
             \right)=
             \left(
               \begin{array}{c}
                 0 \\
                 B_{\tau_{k}}\hat{v} \\
               \end{array}
             \right)
\end{equation*}
with
\begin{equation*}
  D(\mathcal{A})=\mathbb{R}\times D(A),
\end{equation*}
and the map $F: \overline{D(\mathcal{A})}\rightarrow\mathcal{X}$ defined by
\begin{equation*}
  F\left(
               \begin{array}{c}
                 \hat{\tau} \\
                 \hat{v} \\
               \end{array}
             \right)=\left(
                       \begin{array}{c}
                         0 \\
                         W\left(
               \begin{array}{c}
                 \hat{\tau} \\
                 \hat{v} \\
               \end{array}
             \right) \\
                       \end{array}
                     \right),
\end{equation*}
where $W: \overline{D(\mathcal{A})}\rightarrow X$ is defined by
\begin{equation*}
  W\left(
               \begin{array}{c}
                 \hat{\tau} \\
                 \hat{v} \\
               \end{array}
             \right):=(\hat{\tau}+\tau_{k})[H(\hat{v}+\overline{v}_{\hat{\tau}+\tau_{k}})-H(\overline{v}_{\hat{\tau}+\tau_{k}})]-\tau_{k}DH(\overline{v}_{\tau_{k}})(\hat{v})-\hat{\tau}G(\hat{v}),
\end{equation*}
where
\begin{equation*}
  G\left(
     \begin{array}{c}
       0_{\mathbb{R}^{2}} \\
       \varphi \\
     \end{array}
   \right)
  =\left(
               \begin{array}{c}
                 0_{\mathbb{R}^{2}} \\
                 Q\varphi \\
               \end{array}
             \right) \mbox{ and }
             \left(
               \begin{array}{c}
                 0_{\mathbb{R}^{2}} \\
                 \varphi \\
               \end{array}
             \right)=\hat{v}.
\end{equation*}
Then we can yield
\begin{equation*}
     \begin{array}{ccc}
       F\left(
          \begin{array}{c}
            0 \\
            0 \\
          \end{array}
        \right)=0
        & \mbox{ and } & DF\left(
                             \begin{array}{c}
                               0 \\
                               0 \\
                             \end{array}
                           \right)=0.\\
     \end{array}
\end{equation*}
Let $w(t)=\left(
            \begin{array}{c}
              \hat{\tau}(t) \\
              \hat{v}(t) \\
            \end{array}
          \right)
$. Now we can reformulate system (\ref{dhattaudhatx}) as the following system
\begin{equation}\label{dwtdt}
  \frac{dw(t)}{dt}=\mathcal{A}w(t)+F(w(t)), w(0)=w_{0}\in\overline{D(\mathcal{A})}.
\end{equation}

\subsection{Spectral Decomposition}
\noindent

Now we compute the projectors on the generalized eigenspace associated to eigenvalues $\lambda_{0}= i\omega_{k}$ or $-i\omega_{k}$ of $B_{\tau_{k}}$. In terms of the above discussion, we deduced that $\lambda_{0}$ is a pole of $(\lambda I-B_{\tau_{k}})^{-1}$ of finite order $1$. This means that $\lambda_{0}$ is isolated in $\sigma(B_{\tau_{k}})\cap\Omega$, and the Laurent's expansion of the resolvent around $\lambda_{0}$ takes the following form
\begin{equation*}
  (\lambda I-B_{\tau_{k}})^{-1}=\sum_{n=-1}^{+\infty}{(\lambda-\lambda_{0})^{n}B_{n,\lambda_{0}}^{B_{\tau_{k}}}}.
\end{equation*}
The bounded linear operator $B_{-1,\lambda_{0}}^{B_{\tau_{k}}}$ is the projector on the generalized eigenspace of $B_{\tau_{k}}$ associated to $\lambda_{0}$. We remark that
\begin{equation*}
 (\lambda-\lambda_{0})(\lambda I-B_{\tau_{k}})^{-1}=\sum_{m=0}^{+\infty}{(\lambda-\lambda_{0})^{m}B_{m-1,\lambda_{0}}^{B_{\tau_{k}}}}.
\end{equation*}
Therefore, we have the following formula
\begin{equation*}
  B_{-1,\lambda_{0}}^{B_{\tau_{k}}}=\lim\limits_{\lambda\to \lambda_{0}} {(\lambda-\lambda_{0})(\lambda I-B_{\tau_{k}})^{-1}}.
\end{equation*}

\begin{lemma}
  Let Assumption \ref{assumption1} and \ref{assumption2} be satisfied. Then $\lambda_{0}$ is a pole of $(\lambda I-B_{\tau_{k}})^{-1}$ of order $1$, and the projector on the generalized eigenspace of $B_{\tau_{k}}$ associated to the eigenvalue $\lambda_{0}$ is given by
\begin{equation*}
B_{-1,\lambda_{0}}^{B_{\tau_{k}}}\left(
                             \begin{array}{c}
                               \delta \\
                               \varphi \\
                             \end{array}
                           \right)=\left(
                \begin{array}{c}
                  0_{\mathbb{R}^{2}} \\
                  \psi_{\lambda_{0}}\\
                \end{array}
              \right),\\
\end{equation*}
where
\begin{equation}\label{psia}
     \begin{array}{ccl}
       \psi_{\lambda_{0}}(a) & = &\left(\frac{d}{d\lambda}|\Delta(\lambda_{0})|\right)^{-1}e^{-(\lambda_{0} I+\tau_{k} Q)a}\left(\Delta(\lambda_{0})\right)^{*} \\
       & & \times \left[\delta+\tau_{k} DB(\overline{u}_{\tau_{k}})\left(\int_{0}^{a}{e^{-\int_{s}^{a}{(\lambda_{0} I+\tau_{k} Q)dl}}\varphi(s)}ds\right)\right] \\
     \end{array}
\end{equation}
with
\begin{equation}\label{Deltalambdaxing}
\left(\Delta(\lambda)\right)^{*}=\left(
                                                    \begin{array}{cc}
                                                      1+\frac{\xi\tau}{(\lambda+\mu\tau)(\xi+1)} & \frac{\xi\tau}{(\lambda+\mu\tau)(\xi+1)} \\
                                                      \frac{\tau(\eta-1)}{\lambda+\tau}+\frac{\xi\tau e^{-\lambda}}{(\lambda+\tau)(\xi+1)} & 1-\frac{\tau}{\lambda+\tau}+\frac{\xi\tau e^{-\lambda}}{(\lambda+\tau)(\xi+1)} \\
                                                    \end{array}
                                                  \right).
\end{equation}
\end{lemma}
\begin{proof}
Due to
\begin{equation*}
\begin{array}{l}
       (\lambda-\lambda_{0})(\lambda I -B_{\tau_{k}})^{-1}\left(
                                \begin{array}{c}
                                  \delta \\
                                  \varphi \\
                                \end{array}
                              \right)
                              =\left(
                                 \begin{array}{c}
                                   0_{\mathbb{R}^{2}} \\
                                   \psi_{\lambda_{0}} \\
                                 \end{array}
                               \right)\Longleftrightarrow \\
        \psi(a)= (\lambda-\lambda_{0})e^{-\int_{0}^{a}{(\lambda I+\tau_{k} Q)dl}}(\Delta(\lambda))^{-1}\left[\delta+K(\lambda,\varphi)\right]+(\lambda-\lambda_{0})\int_{0}^{a}
    {e^{-\int_{s}^{a}{(\lambda I+\tau_{k} Q)dl}}}\varphi(s)ds, \\
     \end{array}
\end{equation*}
where $\Delta(\lambda)$ and $K(\lambda,\varphi)$ are defined in (\ref{Deltalambda}) and (\ref{Klambdapsi}). We first have $(\Delta(\lambda))^{-1}=\frac{\left(\Delta(\lambda)\right)^{*}}{|\Delta(\lambda)|}$, where $\left(\Delta(\lambda)\right)^{*}$ is given by (\ref{Deltalambdaxing}).
Using $f(\lambda_{0})=0$, we find
\begin{equation*}
  (\lambda_{0}+\mu\tau_{k})\tau_{k}\xi e^{-\lambda_{0}}=-(\xi+1)\lambda_{0}^{2}-\tau_{k}(\mu\xi+\mu+\xi)\lambda_{0}+\xi\tau_{k}^{2}(\eta-1).
\end{equation*}
Furthermore, we obtain
\begin{equation*}

\end{equation*}
Therefore, we get
\begin{equation*}
B_{-1,\lambda_{0}}^{B_{\tau_{k}}}\left(
                             \begin{array}{c}
                               \delta \\
                               \varphi \\
                             \end{array}
                           \right)=\left(
                \begin{array}{c}
                  0_{\mathbb{R}^{2}} \\
                  \psi_{\lambda_{0}}\\
                \end{array}
              \right),\\
\end{equation*}
where $\psi$ is defined in (\ref{psia}).
\end{proof}
\noindent
On the basis of the above results, we obtain a state space decomposition regarding the spectral properties of the linear operator $B_{\tau_{k}}$. More precisely, the projector on the center manifold is defined by
\begin{equation*}
  \hat{\Pi}_{c}\left(
                 \begin{array}{c}
                   \delta \\
                   \varphi \\
                 \end{array}
               \right)=\hat{\Pi}_{i\omega_{k}}\left(
                 \begin{array}{c}
                   \delta \\
                   \varphi \\
                 \end{array}
               \right)+\hat{\Pi}_{-i\omega_{k}}\left(
                 \begin{array}{c}
                   \delta \\
                   \varphi \\
                 \end{array}%
          \right)
\end{equation*}
with
\begin{equation*}
  \psi_{01}=\left(\frac{d}{d\lambda}|\Delta(\lambda_{0})|\right)^{-1}\frac{\xi\tau_{k}}{(\lambda_{0}+\mu\tau_{k})(\xi+1)}.
\end{equation*}
Consequently, on account of the above discussion, we have
\begin{equation}\label{hatPic10}
\hat{\Pi}_{c}\left(
  %
\end{equation}
Then the set $\left\{\left(
                  %
                        \right)
\right\}$ is a basis of $X_{c}:=\hat{\Pi}_{c}(X)$.
By construction, we can yield
\begin{equation*}
     %
\end{equation*}
where
\begin{equation}\label{ppm}
p_{\pm}=\frac{\eta\tau_{k}-(\pm i\omega_{k}+\tau_{k})}{\pm i\omega_{k}+\tau_{k}}.
\end{equation}
We thus get
\begin{equation*}
     %
\end{equation*}
and we obtain that the matrix of $B_{\tau}|_{\hat{\Pi}_{c}(X)}$ regarding basis $\left\{\left(
                  %
\end{equation*}
\begin{lemma}
Let Assumption \ref{assumption1} and \ref{assumption2} be satisfied. Then we have
\begin{equation}\label{hatPih10}
  \hat{\Pi}_{h}\left(
                 %
               \right).
\end{equation}
\end{lemma}
\noindent
Let $\lambda\in i\mathbb{R}\setminus\{-i\omega_{k},i\omega_{k}\}$. Note that for each $\lambda\in\rho(B_{\tau_{k}})$, we have
\begin{equation}\label{lambdaIBtau}
  (\lambda I-B_{\tau_{k}})^{-1}\left(
                      %
                    \right)
\end{equation}
Hence, in accordance with (\ref{hatPih10}) and (\ref{lambdaIBtau}), we can yield
\begin{equation}\label{lambdaIBtaukPih10}
     %
\end{equation}
In addition, in the light of (\ref{hatPih01}) and (\ref{lambdaIBtau}), we have
\begin{equation}\label{lambdaIBtaukPih01}
     %
\end{equation}
According to (\ref{psia}), on the one hand, we have
\begin{equation*}
     %
\end{equation*}
On the other hand, we have
\begin{equation*}
     %
\end{equation*}
Moreover, we obtain
\begin{equation}\label{hatPicb1mub2}
   \hat{\Pi}_{c}\left(
           %
         \right),
\end{equation}
where
\begin{equation*}
  \psi_{h 121}=(1-\psi_{i\omega 12})b_{1}-\psi_{-i\omega 12}b_{3} \\
\end{equation*}
and
\begin{equation*}
  \psi_{h 122}= (\mu-\psi_{i\omega 12})b_{2}-\psi_{-i\omega 12}b_{4}.\\
\end{equation*}
Accordingly, according to (\ref{psia}), we obtain
\begin{equation*}
     %
\end{equation*}
Then, we get
\begin{equation}\label{hatPicb3mub4}
   \hat{\Pi}_{c}\left(
           %
             \right)
              \\
           \end{array}
         \right),
\end{equation}
where
\begin{equation*}
  \psi_{h 341}=(1-\psi_{-i\omega 34})b_{3}-\psi_{i\omega 34}b_{1} \\
\end{equation*}
and
\begin{equation*}
  \psi_{h 342}=(\mu-\psi_{-i\omega 34})b_{4}-\psi_{i\omega 34}b_{2}.\\
\end{equation*}
\begin{lemma}
Let Assumption \ref{assumption1} be satisfied. Then
\begin{equation*}
  \sigma(\mathcal{A})=\sigma(B_{\tau_{k}})\cup \{0\}.
\end{equation*}
Moreover, for $\lambda\in\rho(\mathcal{A})\cap\Omega=\Omega\setminus(\sigma(B_{\tau_{k}})\cup \{0\})$, we have
\begin{equation*}
  (\lambda-\mathcal{A})^{-1}\left(

                            \right)
\end{equation*}
and the eigenvalues $0$ and $\pm i \omega_{k}$ of $\mathcal{A}$ are simple. The corresponding projectors $\Pi_{0}$, $\Pi_{\pm i \omega_{k}}:\mathcal{X}\rightarrow\mathcal{X}$ are defined by
\begin{equation*}
     %
\end{equation*}
\end{lemma}
\noindent
Note that we have
\begin{equation*}
    \begin{array}{cc}
      \overline{\Pi_{i\omega_{k}}(x)}=\Pi_{-i\omega_{k}}(\overline{x}), & \forall x\in\mathcal{X}. \\
    \end{array}
\end{equation*}
In this context, the projectors $\Pi_{c}, \Pi_{h}: \mathcal{X}\rightarrow\mathcal{X}$ are defined by
\begin{equation*}
     \begin{array}{cll}
       \Pi_{c}(x) & :=(\Pi_{0}+\Pi_{i\omega_{k}}+\Pi_{-i\omega_{k}})(x), & \forall x\in\mathcal{X}, \\
     \end{array}
\end{equation*}
and
\begin{equation*}
\begin{array}{cll}
\Pi_{h}(x) & :=(I-\Pi_{c})(x), & \forall x \in \mathcal{X}. \\
\end{array}
\end{equation*}
We denote
\begin{equation*}
     \begin{array}{ccccc}
       \mathcal{X}_{c}:=\Pi_{c}(\mathcal{X}), & \mathcal{X}_{h}:= \Pi_{h}(\mathcal{X}),& \mathcal{A}_{c}:=\mathcal{A}|_{\mathcal{X}_{c}} &\mbox{and} & \mathcal{A}_{h}:=\mathcal{A}|_{\mathcal{X}_{h}}. \\
     \end{array}
\end{equation*}
Now we have the decomposition
\begin{equation*}
  \mathcal{X}=\mathcal{X}_{c}\oplus\mathcal{X}_{h}.
\end{equation*}
Define the basis of $\mathcal{X}_{c}$ by
\begin{equation*}
     %
\end{equation*}
where $c_{1}, c_{2}, c_{3}, c_{4}$, is given by (\ref{c1c2c3c4}).
We can readily check the following lemma.
\begin{lemma}
For $\lambda\in i\mathbb{R}$, we have
\begin{equation*}
  (\lambda I-\mathcal{A}_{h})^{-1}\left(
                                    %
                                  \right).
\end{equation*}
\end{lemma}
\noindent
To simplify the computation, we use the eigenfunctions of $\mathcal{A}$ in $\mathcal{X}_{c}$ and consider
\begin{equation*}
     %
\end{equation*}

\subsection{Computation of Taylor expansion}
\noindent

For convenient of computation, we first give the following basic facts
\begin{equation*}
\overline{v}_{\tau_{k}}=\left(
                  %
\right)
\right\}$ is the basis of $X_{c}$. Set
\begin{equation*}
\hat{v}_{c}=\left(
  %
\right).
\end{equation*}
We observe that for each
\begin{equation*}
  w_{1}:=\left(
           %
                                                 \right)
     \right)\\
     & = & \tau_{k}D^{2}H(\overline{v}_{\tau_{k}})(v_{1},v_{2})+\hat{\tau}_{2}DH(\overline{v}_{\tau_{k}})(v_{1})+ \hat{\tau}_{1}DH(\overline{v}_{\tau_{k}})(v_{2})\\
     &  & +\hat{\tau}_{2}\tau_{k}D^{2}H(\overline{v}_{\tau_{k}})(v_{1},\frac{d\overline{v}_{\hat{\tau}+\tau_{k}}}{d\hat{\tau}}\big|_{\hat{\tau}=0}) +\hat{\tau}_{1}\tau_{k}D^{2}H(\overline{v}_{\tau_{k}})(v_{2},\frac{d\overline{v}_{\hat{\tau}+\tau_{k}}}{d\hat{\tau}}\big|_{\hat{\tau}=0})\\ &&-\hat{\tau}_{1}DG(0)(v_{2})- \hat{\tau}_{2}DG(0)(v_{1}),\\
  \end{array}
\end{equation}
where
\begin{equation*}
  DG(0)\left(

\end{equation*}
and
\begin{equation*}
    \tilde{\tilde{\psi}}_{2} = -\tilde{\tilde{\psi}}_{1}+\hat{\tau}\eta\int_{0}^{+\infty}{\hat{\varphi}^{1}(a)da}. \\
\end{equation*}
By projecting on $X_{c}$, we obtain
\begin{equation*}
     %
\end{equation*}
Now we compute $\frac{1}{2!}D^{2}W(0)(w_{c})^{2}$, $\frac{1}{2!}\Pi_{c}D^{2}F(0)(w_{c})^{2}$, $\frac{1}{2!}\Pi_{h}D^{2}F(0)(w_{c})^{2}$, and $\frac{1}{3!}D^{3}W(0)(w_{c})^{3}$, expressed in terms of the basis $\{\hat{e}_{1}, \hat{e}_{2}, \hat{e}_{3}\}$. We first obtain that
\begin{equation*}
%
\end{equation*}
and
\begin{equation*}
    \tilde{\psi}_{2}  = -\tilde{\psi}_{1}+\hat{\tau}\eta\int_{0}^{+\infty}{(x_{1}b_{1}+x_{2}b_{3})(a)da}.
\end{equation*}
Thus, according to (\ref{12D2W0w2}), (\ref{hatPic10}), (\ref{hatPic01}), (\ref{hatPicb1mub2}) and (\ref{hatPicb3mub4}), we have
\begin{equation}\label{PicD2F0}
  %
\end{equation*}
In addition, according to (\ref{hatPih10}), (\ref{hatPih01}), (\ref{hatPihb1mub2}) and (\ref{hatPihb3mub4}), we have
\begin{equation}\label{PihD2F0}
     %
\end{equation*}
Moreover, we have
\begin{equation*}
     %
                             \right)
         \right) \\
        & = & \hat{\tau}_{1}D^{2}H(\overline{v}_{\tau_{k}})(v_{2},v_{3})+\hat{\tau}_{2}D^{2}H(\overline{v}_{\tau_{k}})(v_{1},v_{3})
        +\hat{\tau}_{3}D^{2}H(\overline{v}_{\tau_{k}})(v_{1},v_{2}) \\
        &&+2\hat{\tau}_{2}\hat{\tau}_{3}D^{2}H(\overline{v}_{\tau_{k}})\left(v_{1},\frac{d\overline{v}_{\hat{\tau}+\tau_{k}}}{d\hat{\tau}}\big|_{\hat{\tau}=0} \right)
        +2\hat{\tau}_{1}\hat{\tau}_{3}D^{2}H(\overline{v}_{\tau_{k}})\left(v_{2},\frac{d\overline{v}_{\hat{\tau}+\tau_{k}}}{d\hat{\tau}}\big|_{\hat{\tau}=0} \right)\\
        &&+2\hat{\tau}_{1}\hat{\tau}_{2}D^{2}H(\overline{v}_{\tau_{k}})\left(v_{3},\frac{d\overline{v}_{\hat{\tau}+\tau_{k}}}{d\hat{\tau}}\big|_{\hat{\tau}=0} \right)\\
        &&+\hat{\tau}_{2}\hat{\tau}_{3}\tau_{k}D^{2}H(\overline{v}_{\tau_{k}})\left(v_{1},\frac{d^{2}\overline{v}_{\hat{\tau}+\tau_{k}}}{d(\hat{\tau})^{2}}\big|_{\hat{\tau}=0} \right)
        +\hat{\tau}_{1}\hat{\tau}_{3}\tau_{k}D^{2}H(\overline{v}_{\tau_{k}})\left(v_{2},\frac{d^{2}\overline{v}_{\hat{\tau}+\tau_{k}}}{d(\hat{\tau})^{2}}\big|_{\hat{\tau}=0} \right)\\
        &&+\hat{\tau}_{1}\hat{\tau}_{2}\tau_{k}D^{2}H(\overline{v}_{\tau_{k}})\left(v_{3},\frac{d^{2}\overline{v}_{\hat{\tau}+\tau_{k}}}{d(\hat{\tau})^{2}}\big|_{\hat{\tau}=0} \right)\\
        &&+\tau_{k}D^{3}H(\overline{v}_{\tau_{k}})(v_{1},v_{2},v_{3})\\
        &&+\hat{\tau}_{3}\tau_{k}D^{3}H(\overline{v}_{\tau_{k}})\left(v_{1},v_{2},\frac{d\overline{v}_{\hat{\tau}+\tau_{k}}}{d\hat{\tau}}\big|_{\hat{\tau}=0}\right)
        +\hat{\tau}_{2}\tau_{k}D^{3}H(\overline{v}_{\tau_{k}})\left(v_{1},v_{3},\frac{d\overline{v}_{\hat{\tau}+\tau_{k}}}{d\hat{\tau}}\big|_{\hat{\tau}=0}\right)\\
        &&+\hat{\tau}_{1}\tau_{k}D^{3}H(\overline{v}_{\tau_{k}})\left(v_{2},v_{3},\frac{d\overline{v}_{\hat{\tau}+\tau_{k}}}{d\hat{\tau}}\big|_{\hat{\tau}=0}\right)\\
        &&+\hat{\tau}_{2}\hat{\tau}_{3}\tau_{k}D^{3}H(\overline{v}_{\tau_{k}})\left(v_{1},\frac{d\overline{v}_{\hat{\tau}+\tau_{k}}}{d\hat{\tau}}\big|_{\hat{\tau}=0},\frac{d\overline{v}_{\hat{\tau}+\tau_{k}}}{d\hat{\tau}}\big|_{\hat{\tau}=0}\right)\\
        &&+\hat{\tau}_{1}\hat{\tau}_{3}\tau_{k}D^{3}H(\overline{v}_{\tau_{k}})\left(v_{2},\frac{d\overline{v}_{\hat{\tau}+\tau_{k}}}{d\hat{\tau}}\big|_{\hat{\tau}=0},\frac{d\overline{v}_{\hat{\tau}+\tau_{k}}}{d\hat{\tau}}\big|_{\hat{\tau}=0}\right)\\
        &&+\hat{\tau}_{1}\hat{\tau}_{2}\tau_{k}D^{3}H(\overline{v}_{\tau_{k}})\left(v_{3},\frac{d\overline{v}_{\hat{\tau}+\tau_{k}}}{d\hat{\tau}}\big|_{\hat{\tau}=0},\frac{d\overline{v}_{\hat{\tau}+\tau_{k}}}{d\hat{\tau}}\big|_{\hat{\tau}=0}\right),\\
     \end{array}
\end{equation*}
where
\begin{equation*}
    \begin{array}{cc}
      D^{3}H(\overline{v}_{\tau_{k}})\left(\left(
                                        \begin{array}{c}
                                          0_{\mathbb{R}^{2}} \\
                                          \left(
                                            \begin{array}{c}
                                              \varphi_{1}^{1} \\
                                              \varphi_{1}^{2} \\
                                            \end{array}
                                          \right)
                                           \\
                                        \end{array}
                                      \right)
      ,\left(
         \begin{array}{c}
           0_{\mathbb{R}^{2}} \\
           \left(
             \begin{array}{c}
               \varphi_{2}^{1} \\
               \varphi_{2}^{2} \\
             \end{array}
           \right)
            \\
         \end{array}
       \right),\left(
         \begin{array}{c}
           0_{\mathbb{R}^{2}} \\
           \left(
             \begin{array}{c}
               \varphi_{3}^{1} \\
               \varphi_{3}^{2} \\
             \end{array}
           \right)
            \\
         \end{array}
       \right)\right)=\left(
                                \begin{array}{c}
                                 \left(
                                    \begin{array}{c}
                                      \psi_{D3H1} \\
                                      \psi_{D3H2} \\
                                    \end{array}
                                  \right)
                                   \\
                                  0_{L^{1}} \\
                                \end{array}
                              \right)
    \end{array}
\end{equation*}
with
\begin{equation*}
     \begin{array}{ccl}
       \psi_{D3H1} & = & -\frac{1}{(1+\xi)^{2}}  \big\{\int_{0}^{+\infty}{\varphi_{1}^{1}(a)da}\int_{0}^{+\infty}{\varphi_{2}^{2}(a)da}\int_{0}^{+\infty}{\beta(a)\varphi_{3}^{1}(a)da} \\
       &&+\int_{0}^{+\infty}{\varphi_{1}^{1}(a)da}\int_{0}^{+\infty}{\varphi_{3}^{2}(a)da}\int_{0}^{+\infty}{\beta(a)\varphi_{2}^{1}(a)da} \\
        &&+\int_{0}^{+\infty}{\varphi_{1}^{2}(a)da}\int_{0}^{+\infty}{\varphi_{2}^{1}(a)da}\int_{0}^{+\infty}{\beta(a)\varphi_{3}^{1}(a)da}\\
        &&+\int_{0}^{+\infty}{\varphi_{1}^{2}(a)da}\int_{0}^{+\infty}{\varphi_{3}^{1}(a)da}\int_{0}^{+\infty}{\beta(a)\varphi_{2}^{1}(a)da} \\
        &&+\int_{0}^{+\infty}{\varphi_{2}^{1}(a)da}\int_{0}^{+\infty}{\varphi_{3}^{2}(a)da}\int_{0}^{+\infty}{\beta(a)\varphi_{1}^{1}(a)da} \\
        &&+\int_{0}^{+\infty}{\varphi_{2}^{2}(a)da}\int_{0}^{+\infty}{\varphi_{3}^{1}(a)da}\int_{0}^{+\infty}{\beta(a)\varphi_{1}^{1}(a)da}\\
        &&-2\int_{0}^{+\infty}{\varphi_{1}^{1}(a)da}\int_{0}^{+\infty}{\beta(a)\varphi_{2}^{1}(a)da}\int_{0}^{+\infty}{\beta(a)\varphi_{3}^{1}(a)da}\\
        &&-2\int_{0}^{+\infty}{\varphi_{2}^{1}(a)da}\int_{0}^{+\infty}{\beta(a)\varphi_{1}^{1}(a)da}\int_{0}^{+\infty}{\beta(a)\varphi_{3}^{1}(a)da}\\
         &&-2\int_{0}^{+\infty}{\varphi_{3}^{1}(a)da}\int_{0}^{+\infty}{\beta(a)\varphi_{1}^{1}(a)da}\int_{0}^{+\infty}{\beta(a)\varphi_{2}^{1}(a)da}\big\}\\
        &&-\frac{\xi}{(1+\xi)^{3}}\big\{6\int_{0}^{+\infty}{\beta(a)\varphi_{1}^{1}(a)da}\int_{0}^{+\infty}{\beta(a)\varphi_{2}^{1}(a)da}\int_{0}^{+\infty}{\beta(a)\varphi_{3}^{1}(a)da}\\
        &&-2\int_{0}^{+\infty}{\varphi_{1}^{2}(a)da}\int_{0}^{+\infty}{\beta(a)\varphi_{2}^{1}(a)da}\int_{0}^{+\infty}{\beta(a)\varphi_{3}^{1}(a)da}\\
        &&-2\int_{0}^{+\infty}{\varphi_{2}^{2}(a)da}\int_{0}^{+\infty}{\beta(a)\varphi_{1}^{1}(a)da}\int_{0}^{+\infty}{\beta(a)\varphi_{3}^{1}(a)da}\\
        &&-2\int_{0}^{+\infty}{\varphi_{3}^{2}(a)da}\int_{0}^{+\infty}{\beta(a)\varphi_{1}^{1}(a)da}\int_{0}^{+\infty}{\beta(a)\varphi_{2}^{1}(a)da}\big\}
     \end{array}
\end{equation*}
and
\begin{equation*}
\psi_{D3H2}=-\psi_{D3H1}.
\end{equation*}
Therefore, we obtain
\begin{equation*}
  \frac{1}{3!}D^{3}W(0)(w_{c})^{3}=\frac{1}{3!}D^{3}W(0)\left(
                                                                                           \begin{array}{c}
                                                                                             \hat{\tau} \\
                                                                                             \hat{v}_{c} \\
                                                                                           \end{array}
                                                                                         \right)
    ^{3}=
    \left(
      \begin{array}{c}
        \left(
          \begin{array}{c}
            \hat{\psi}_{1} \\
            \hat{\psi}_{2} \\
          \end{array}
        \right)\\
        0_{L^{1}} \\
      \end{array}
    \right)
\end{equation*}
with
\begin{equation*}
     \begin{array}{ccl}
       \hat{\psi}_{1} & = &
       \bigg\{
       \frac{(\hat{\tau})^{2}}{\xi+1}\left(\frac{\xi(\tau_{k}-1)}{\tau_{k}}-\frac{[\Lambda+\xi(\eta-1)](\mu\tau_{k}-1)}{\mu\tau_{k}}\right)
       +\frac{(\hat{\tau})^{2}\xi}{(\xi+1)^{2}}
       \left(\frac{\xi(\tau_{k}-1)^{2}}{\tau_{k}}-\frac{[\Lambda+\xi(\eta-1)](\mu\tau_{k}-1)(\tau_{k}-1)}{\mu\tau_{k}}\right)\\
       &&-\frac{(\hat{\tau})^{2}\big\{\xi(\tau_{k}-2)+[\Lambda+\xi(\eta-1)](2-\mu\tau_{k})\big\}}{2(\xi+1)}
     \bigg\}\int_{0}^{+\infty}{(x_{1}b_{1}+x_{2}b_{3})(a)da}\\

     &&-\frac{(\hat{\tau})^{2}\xi[(\xi-1)\tau_{k}^{2}+4\tau_{k}-2]}{2\tau_{k}(\xi+1)^{3}}
     \int_{0}^{+\infty}{(x_{1}b_{2}+x_{2}b_{4})(a)da}\\

     &&+\frac{\hat{\tau}}{\xi+1}\left(\frac{\xi(\tau_{k}-1)}{\xi+1}+1\right)\int_{0}^{+\infty}
     {(x_{1}b_{1}+x_{2}b_{3})(a)da}\int_{0}^{+\infty}{(x_{1}b_{2}+x_{2}b_{4})(a)da}\\

     &&-\frac{\tau_{k}\xi}{(\xi+1)^{3}}\left(\int_{0}^{+\infty}{\beta(a)(x_{1}b_{1}+x_{2}b_{3})(a)da}\right)^{3}\\

     &&+\bigg\{\frac{\hat{\tau}\xi(2-\tau_{k})}{(\xi+1)^{2}}
     +\frac{\hat{\tau}\xi}{(\xi+1)^{3}}\left(3\xi(\tau_{k}-1)-\frac{[\Lambda+\xi(\eta-1)](\mu\tau_{k}-1)}{\mu}\right)\bigg\}\\
     &&\times\left(\int_{0}^{+\infty}{\beta(a)(x_{1}b_{1}+x_{2}b_{3})(a)da}\right)^{2}\\

     &&+\bigg\{\frac{\hat{\tau}^{2}\xi}{(\xi+1)^{2}}\left(\frac{2\xi(\tau_{k}-1)^{2}}{\tau_{k}}-\frac{[\Lambda+\xi(\eta-1)](\mu\tau_{k}-1)(\tau_{k}-1)}{\mu\tau_{k}}\right)\\
     &&-\frac{\hat{\tau}^{2}\xi^{2}}{(\xi+1)^{3}}\left(\frac{3\xi(\tau_{k}-1)^{2}}{\tau_{k}}-\frac{2[\Lambda+\xi(\eta-1)](\mu\tau_{k}-1)(\tau_{k}-1)}{\mu\tau_{k}}\right)\\
     &&-\frac{\hat{\tau}^{2}\xi}{(\xi+1)^{2}}\left(\frac{2\xi(\tau_{k}-1)}{\tau_{k}}-\frac{[\Lambda+\xi(\eta-1)](\mu\tau_{k}-1)}{\mu\tau_{k}}\right)+\frac{\hat{\tau}^{2}\xi(\tau_{k}-1)}{\tau_{k}(\xi+1)}\\
     &&+\frac{\hat{\tau}^{2}\xi\big\{2\xi(\tau_{k}-2)+[\Lambda+\xi(\eta-1)](2-\mu\tau_{k})\big\}}{2(\xi+1)^{2}}-\frac{\hat{\tau}^{2}\xi(\tau_{k}-2)}{2(\xi+1)}
     \bigg\}\int_{0}^{+\infty}{\beta(a)(x_{1}b_{1}+x_{2}b_{3})(a)da}\\

     &&+\bigg\{\frac{\hat{\tau}}{(\xi+1)^{2}}\left(2\xi(1-\tau_{k})+\frac{[\Lambda+\xi(\eta-1)](\mu\tau_{k}-1)}{\mu}\right)-\frac{\hat{\tau}}{\xi+1}\bigg\}\\
     &&\times\int_{0}^{+\infty}{(x_{1}b_{1}+x_{2}b_{3})(a)da}\int_{0}^{+\infty}{\beta(a)(x_{1}b_{1}+x_{2}b_{3})(a)da}\\
     &&+\left(\frac{\hat{\tau}\xi(1-\xi)(\tau_{k}-1)}{(\xi+1)^{3}}-\frac{\hat{\tau}\xi}{(\xi+1)^{2}}\right)\int_{0}^{+\infty}{(x_{1}b_{2}+x_{2}b_{4})(a)da}\int_{0}^{+\infty}{\beta(a)(x_{1}b_{1}+x_{2}b_{3})(a)da}\\
     &&+\frac{\tau_{k}}{(\xi+1)^{2}}\int_{0}^{+\infty}{(x_{1}b_{1}+x_{2}b_{3})(a)da}\left(\int_{0}^{+\infty}{\beta(a)(x_{1}b_{1}+x_{2}b_{3})(a)da}\right)^{2}\\
     &&+\frac{\xi\tau_{k}}{(\xi+1)^{3}}\int_{0}^{+\infty}{(x_{1}b_{2}+x_{2}b_{4})(a)da}\left(\int_{0}^{+\infty}{\beta(a)(x_{1}b_{1}+x_{2}b_{3})(a)da}\right)^{2}\\
     &&-\frac{\tau_{k}}{(\xi+1)^{2}}\int_{0}^{+\infty}{(x_{1}b_{1}+x_{2}b_{3})(a)da}
     \int_{0}^{+\infty}{(x_{1}b_{2}+x_{2}b_{4})(a)da}\int_{0}^{+\infty}{\beta(a)(x_{1}b_{1}+x_{2}b_{3})(a)da}.\\
     \end{array}
\end{equation*}
and
\begin{equation*}
\hat{\psi}_{2}  = -\hat{\psi}_{1}.
\end{equation*}
Furthermore, we can yield
\begin{equation}\label{13PicD3F0wc3}
     \begin{array}{ccl}
     &&\frac{1}{3!}\Pi_{c}D^{3}F(0)(w_{c})^{3}\\
        & = & \left(
                                                       \begin{array}{c}
                                                         0 \\
                                                         \frac{1}{3!}\hat{\Pi}_{c}D^{3}W(0)(w_{c})^{3} \\
                                                       \end{array}
                                                     \right)\\
        & = & \left(
                                                       \begin{array}{c}
                                                         0_{\mathbb{R}} \\
                                                         \hat{\psi}_{1}\hat{\Pi}_{c}\left(
                                                                                  \begin{array}{c}
                                                                                    \left(
                                                                                      \begin{array}{c}
                                                                                        1 \\
                                                                                        0 \\
                                                                                      \end{array}
                                                                                    \right)
                                                                                     \\
                                                                                    0_{L^{1}} \\
                                                                                  \end{array}
                                                                                \right)
                                                          \\
                                                       \end{array}
                                                     \right)+
                                                     \left(
                                                       \begin{array}{c}
                                                         0_{\mathbb{R}} \\
                                                         \hat{\psi}_{2}\hat{\Pi}_{c}\left(
                                                                                  \begin{array}{c}
                                                                                    \left(
                                                                                      \begin{array}{c}
                                                                                        0 \\
                                                                                        1 \\
                                                                                      \end{array}
                                                                                    \right)
                                                                                     \\
                                                                                    0_{L^{1}} \\
                                                                                  \end{array}
                                                                                \right)
                                                          \\
                                                       \end{array}
                                                     \right)\\
&=&\left(
     \begin{array}{c}
       0_{\mathbb{R}} \\
       \hat{\psi}_{1}\left(
                       \begin{array}{c}
                         0_{\mathbb{R}^{2}} \\
                         \psi_{c10} \\
                       \end{array}
                     \right)
        \\
     \end{array}
   \right)+\left(
     \begin{array}{c}
       0_{\mathbb{R}} \\
       \hat{\psi}_{2}\left(
                       \begin{array}{c}
                         0_{\mathbb{R}^{2}} \\
                         \psi_{c01} \\
                       \end{array}
                     \right)
        \\
     \end{array}
   \right)\\
&=&\left(
     \begin{array}{c}
       0_{\mathbb{R}} \\
       \left(
         \begin{array}{c}
           0_{\mathbb{R}^{2}} \\
           \hat{\psi}_{1}\left(\frac{d}{d\lambda}|\Delta(i\omega_{k})|\right)^{-1}\left(
                              \begin{array}{c}
                                b_{1} \\
                                b_{2} \\
                              \end{array}
                            \right)+
                            \hat{\psi}_{1}\left(\frac{d}{d\lambda}|\Delta(-i\omega_{k})|\right)^{-1}\left(
                              \begin{array}{c}
                                b_{3} \\
                                b_{4} \\
                              \end{array}
                            \right)
            \\
         \end{array}
       \right)
        \\
     \end{array}
   \right).
     \end{array}
\end{equation}
%where
%\begin{equation*}
%  \begin{array}{ccl}
%    \psi_{\Pi cD3F12} & = & \hat{\psi}_{1}\left(\frac{d}{d\lambda}|\Delta(i\omega_{k})|\right)^{-1}
%       \frac{(i\omega_{k}+\mu\tau_{k})(\xi+1)+\xi\tau_{k}}{(i\omega_{k}+\mu\tau_{k})(\xi+1)} \\
%     &  & +\hat{\psi}_{2}\left(\frac{d}{d\lambda}|\Delta(i\omega_{k})|\right)^{-1}
%       \frac{\xi\tau_{k}}{(i\omega_{k}+\mu\tau_{k})(\xi+1)}, \\
%  \end{array}
%\end{equation*}
%and
%\begin{equation*}
%  \begin{array}{ccl}
%    \psi_{\Pi cD3F34} & = & \hat{\psi}_{1}\left(\frac{d}{d\lambda}|\Delta(-i\omega_{k})|\right)^{-1}
%       \frac{(-i\omega_{k}+\mu\tau_{k})(\xi+1)+\xi\tau_{k}}{(-i\omega_{k}+\mu\tau_{k})(\xi+1)} \\
%     &  & +\hat{\psi}_{2}\left(\frac{d}{d\lambda}|\Delta(-i\omega_{k})|\right)^{-1}
%       \frac{\xi\tau_{k}}{(-i\omega_{k}+\mu\tau_{k})(\xi+1)}. \\
%  \end{array}
%\end{equation*}

\subsection{Computation of $L_{2}$}
\noindent

In the following we apply the method described in Liu et al. [\cite{LiuZhihuaMagalPierreRuanShigui-JDE-2014}, Theorem 4.2] for $k=2$. The main point is to compute $L_{2}\in \mathcal{L}_{s}(\mathcal{X}_{c}^{2},\mathcal{X}_{h}\cap D(\mathcal{A}))$ by solving the following equation for each $(w_{1},w_{2})\in \mathcal{X}_{c}^{2}$ :
\begin{equation}\label{ddtL2}
\frac{d}{dt}[L_{2}(e^{\mathcal{A}_{c}t}w_{1},e^{\mathcal{A}_{c}t}w_{2})](0)=\mathcal{A}_{h}L_{2}(w_{1},w_{2})+\frac{1}{2!}\Pi_{h}D^{2}F(0)(w_{1},w_{2}).
\end{equation}
Note that
\begin{equation*}
\frac{d}{dt}[L_{2}(e^{\mathcal{A}_{c}t}w_{1},e^{\mathcal{A}_{c}t}w_{2})](0)=L_{2}(\mathcal{A}_{c}w_{1},w_{2})+L_{2}(w_{1},\mathcal{A}_{c}w_{2}).
\end{equation*}
So system (\ref{ddtL2}) can be rewritten as
\begin{equation}\label{L2L2}
L_{2}(\mathcal{A}_{c}w_{1},w_{2})+L_{2}(w_{1},\mathcal{A}_{c}w_{2})= \mathcal{A}_{h}L_{2}(w_{1},w_{2})+\frac{1}{2!}\Pi_{h}D^{2}F(0)(w_{1},w_{2}).
\end{equation}

\noindent \textbf{(i) Computation of $L_{2}(\hat{e}_{2},\hat{e}_{2})$}: Since $\mathcal{A}_{c}\hat{e}_{2}=i\omega_{k}\hat{e}_{2}$, the equation
\begin{equation*}
L_{2}(\mathcal{A}_{c}\hat{e}_{2},\hat{e}_{2})+L_{2}(\hat{e}_{2},\mathcal{A}_{c}\hat{e}_{2})= \mathcal{A}_{h}L_{2}(\hat{e}_{2},\hat{e}_{2})+\frac{1}{2!}\Pi_{h}D^{2}F(0)(\hat{e}_{2},\hat{e}_{2})
\end{equation*}
is equivalent to
\begin{equation*}
(2i\omega_{k}-\mathcal{A}_{h})L_{2}(\hat{e}_{2},\hat{e}_{2})= \frac{1}{2!}\Pi_{h}D^{2}F(0)(\hat{e}_{2},\hat{e}_{2}),
\end{equation*}
where
\begin{equation*}
  \begin{array}{ccl}
     &&D^{2}F(0)(\hat{e}_{2},\hat{e}_{2})
      = \left(
             \begin{array}{c}
               0_{\mathbb{R}} \\
               D^{2}W(0)\left(
                                          \begin{array}{c}
                                            0_{\mathbb{R}} \\
                                            \left(
                                              \begin{array}{c}
                                                0_{\mathbb{R}^{2}} \\
                                                \left(
                                                  \begin{array}{c}
                                                    b_{1} \\
                                                    b_{2} \\
                                                  \end{array}
                                                \right)
                                                 \\
                                              \end{array}
                                            \right)
                                             \\
                                          \end{array}
                                        \right)^{2}\\
             \end{array}
           \right)
      = \left(
                     \begin{array}{c}
                       0_{\mathbb{R}} \\
                       \tau_{k}D^{2}H(\overline{v}_{\tau_{k}})\left(
                                                                          \begin{array}{c}
                                                                            0_{\mathbb{R}^{2}} \\
                                                                            \left(
                                                                              \begin{array}{c}
                                                                                b_{1} \\
                                                                                b_{2} \\
                                                                              \end{array}
                                                                            \right)
                                                                             \\
                                                                          \end{array}
                                                                        \right)^{2}\\
                     \end{array}
                   \right).\\
  \end{array}
\end{equation*}
Thus, we have
\begin{equation*}
  D^{2}F(0)(\hat{e}_{2},\hat{e}_{2})=c_{2210}\left(
                                               \begin{array}{c}
                                                 0_{\mathbb{R}} \\
                                                 \left(
                                             \begin{array}{c}
                                                \left(
                                                 \begin{array}{l}
                                                   1 \\
                                                   0 \\
                                                 \end{array}
                                               \right)\\
                                               0_{L^{1}} \\
                                             \end{array}
                                           \right) \\
                                               \end{array}
                                             \right)
                                      -c_{2210}\left(
                                               \begin{array}{c}
                                                 0_{\mathbb{R}} \\
                                                 \left(
                                             \begin{array}{c}
                                                \left(
                                                 \begin{array}{l}
                                                   0 \\
                                                   1 \\
                                                 \end{array}
                                               \right)\\
                                               0_{L^{1}} \\
                                             \end{array}
                                           \right) \\
                                               \end{array}
                                             \right)
\end{equation*}
with
\begin{equation*}
    c_{2210}  =\frac{2\tau_{k} e^{-i\omega_{k}}(\xi+1-\xi e^{-i\omega_{k}})}{(-i\mu\tau_{k}+\omega_{k})(i\omega_{k}+\tau_{k})^{2}(\xi+1)^{2}}\big
    [p_{+}e^{i\omega_{k}}(i\tau_{k}-\omega_{k})-\omega_{k}+i\mu\tau_{k}\big].\\
\end{equation*}
Therefore, on the basis of (\ref{lambdaIBtaukPih10}) and (\ref{lambdaIBtaukPih01}), we have
\begin{equation}\label{hate2hate2}
  \begin{array}{ccl}
L_{2}(\hat{e}_{2},\hat{e}_{2})& = & \frac{1}{2}c_{2210}(2i\omega_{k}-\mathcal{A}_{h})^{-1}\Pi_{h}\left(
                                             \begin{array}{c}
                                               0_{\mathbb{R}} \\
                                               \left(
                                                 \begin{array}{c}
                                                   \left(
                                                     \begin{array}{c}
                                                       1 \\
                                                       0 \\
                                                     \end{array}
                                                   \right)
                                                    \\
                                                   0_{L^{1}} \\
                                                 \end{array}
                                               \right)
                                                \\
                                             \end{array}
                                           \right)\\
&&-\frac{1}{2}c_{2210}(2i\omega_{k}-\mathcal{A}_{h})^{-1}\Pi_{h}\left(
                                             \begin{array}{c}
                                               0_{\mathbb{R}} \\
                                               \left(
                                                 \begin{array}{c}
                                                   \left(
                                                     \begin{array}{c}
                                                       0 \\
                                                       1 \\
                                                     \end{array}
                                                   \right)
                                                    \\
                                                   0_{L^{1}} \\
                                                 \end{array}
                                               \right)
                                                \\
                                             \end{array}
                                           \right)\\
 & = & \frac{1}{2}c_{2210}\left(
 \begin{array}{c}
 0_{\mathbb{R}} \\
 \left(2i\omega_{k}-B_{\tau}|_{\hat{\Pi}_{h}(X)}\right)^{-1}\hat{\Pi}_{h}\left(
 \begin{array}{c}
 \left(
 \begin{array}{c}
 1 \\
 0 \\
 \end{array}
 \right)\\
 0_{L^{1}} \\
 \end{array}
 \right) \\
 \end{array}
 \right)\\
 &&-\frac{1}{2}c_{2210}\left(
 \begin{array}{c}
 0_{\mathbb{R}} \\
 \left(2i\omega_{k}-B_{\tau}|_{\hat{\Pi}_{h}(X)}\right)^{-1}\hat{\Pi}_{h}\left(
 \begin{array}{c}
 \left(
 \begin{array}{c}
 0 \\
 1 \\
 \end{array}
 \right)\\
 0_{L^{1}} \\
 \end{array}
 \right) \\
 \end{array}
 \right)\\
 & = &\left(
        \begin{array}{c}
          0_{\mathbb{R}} \\
          \left(
            \begin{array}{c}
              0_{\mathbb{R}^{2}} \\
               \left(
                 \begin{array}{c}
                   \psi_{2,2,1} \\
                   \psi_{2,2,2} \\
                 \end{array}
               \right)
               \\
            \end{array}
          \right)
           \\
        \end{array}
      \right),
  \end{array}
\end{equation}
where
\begin{equation}\label{psi221}
  \begin{array}{ccl}
    \psi_{2,2,1}& = & \frac{1}{2}c_{2210}\bigg(\frac{1}{|\Delta(2i\omega_{k})|}e^{-(2i\omega_{k}+\tau_{k})\cdot}
     -\left(\frac{d}{d\lambda}|\Delta(i\omega_{k})|\right)^{-1}
     \frac{e^{-(i\omega_{k}+\tau_{k})\cdot}}{i\omega_{k}} \\
     && -\left(\frac{d}{d\lambda}|\Delta(-i\omega_{k})|\right)^{-1}
       \frac{e^{-(-i\omega_{k}+\tau_{k})\cdot}}{3i\omega_{k}}\bigg)\\
  \end{array}
\end{equation}
and
\begin{equation}\label{psi222}
  \begin{array}{ccl}
    \psi_{2,2,2}& = & \frac{1}{2}c_{2210}\bigg(\frac{1}{|\Delta(2i\omega_{k})|}\frac{\eta\tau_{k}-(2i\omega_{k}+\tau_{k})}{2i\omega_{k}+\tau_{k}}
    e^{-(2i\omega_{k}+\mu\tau_{k})\cdot}
     -\left(\frac{d}{d\lambda}|\Delta(i\omega_{k})|\right)^{-1}
     \frac{p_{+}e^{-(i\omega_{k}+\mu\tau_{k})\cdot}}{i\omega_{k}} \\
     && -\left(\frac{d}{d\lambda}|\Delta(-i\omega_{k})|\right)^{-1}
       \frac{p_{-}e^{-(-i\omega_{k}+\mu\tau_{k})\cdot}}{3i\omega_{k}}\bigg).\\
  \end{array}
\end{equation}
\noindent \textbf{(ii) Computation of $L_{2}(\hat{e}_{2},\hat{e}_{3})$}: Since $\mathcal{A}_{c}\hat{e}_{2}=i\omega_{k}\hat{e}_{2}$ and $\mathcal{A}_{c}\hat{e}_{3}=-i\omega_{k}\hat{e}_{3}$, the equation
\begin{equation*}
L_{2}(\mathcal{A}_{c}\hat{e}_{2},\hat{e}_{3})+L_{2}(\hat{e}_{2},\mathcal{A}_{c}\hat{e}_{3})= \mathcal{A}_{h}L_{2}(\hat{e}_{2},\hat{e}_{3})+\frac{1}{2!}\Pi_{h}D^{2}F(0)(\hat{e}_{2},\hat{e}_{3})
\end{equation*}
is equivalent to
\begin{equation*}
(0-\mathcal{A}_{h})L_{2}(\hat{e}_{2},\hat{e}_{3})= \frac{1}{2!}\Pi_{h}D^{2}F(0)(\hat{e}_{2},\hat{e}_{3}),
\end{equation*}
where
\begin{equation*}
  \begin{array}{ccl}
     &&D^{2}F(0)(\hat{e}_{2},\hat{e}_{3})    \\
     & = & \left(
             \begin{array}{c}
               0_{\mathbb{R}} \\
               D^{2}W(0)\left(\left(
                                          \begin{array}{c}
                                            0_{\mathbb{R}} \\
                                            \left(
                                              \begin{array}{c}
                                                0_{\mathbb{R}^{2}} \\
                                                \left(
                                                  \begin{array}{c}
                                                    b_{1} \\
                                                    b_{2} \\
                                                  \end{array}
                                                \right)
                                                 \\
                                              \end{array}
                                            \right)
                                             \\
                                          \end{array}
                                        \right),  \left(
                                          \begin{array}{c}
                                            0_{\mathbb{R}} \\
                                            \left(
                                              \begin{array}{c}
                                                0_{\mathbb{R}^{2}} \\
                                                \left(
                                                  \begin{array}{c}
                                                    b_{3} \\
                                                    b_{4} \\
                                                  \end{array}
                                                \right) \\
                                              \end{array}
                                            \right)
                                             \\
                                          \end{array}
                                        \right)
               \right) \\
             \end{array}
           \right)
      \\
     & = &\left(
                     \begin{array}{c}
                       0_{\mathbb{R}} \\
                       \tau_{k}D^{2}H(\overline{v}_{\tau_{k}})\left(\left(
                                                                          \begin{array}{c}
                                                                            0_{\mathbb{R}^{2}} \\
                                                                            \left(
                                                                              \begin{array}{c}
                                                                                b_{1} \\
                                                                                b_{2} \\
                                                                              \end{array}
                                                                            \right)
                                                                             \\
                                                                          \end{array}
                                                                        \right),\left(
                                                                          \begin{array}{c}
                                                                            0_{\mathbb{R}^{2}} \\
                                                                            \left(
                                                                              \begin{array}{c}
                                                                                b_{3} \\
                                                                                b_{4} \\
                                                                              \end{array}
                                                                            \right) \\
                                                                          \end{array}
                                                                        \right)
                       \right) \\
                     \end{array}
                   \right).\\
  \end{array}
\end{equation*}
Thus, we have
\begin{equation*}
  D^{2}F(0)(\hat{e}_{2},\hat{e}_{3})=c_{2310}\left(
                                               \begin{array}{c}
                                                 0_{\mathbb{R}} \\
                                                 \left(
                                             \begin{array}{c}
                                                \left(
                                                 \begin{array}{l}
                                                   1 \\
                                                   0 \\
                                                 \end{array}
                                               \right)\\
                                               0_{L^{1}} \\
                                             \end{array}
                                           \right) \\
                                               \end{array}
                                             \right)
                                      -c_{2310}\left(
                                               \begin{array}{c}
                                                 0_{\mathbb{R}} \\
                                                 \left(
                                             \begin{array}{c}
                                                \left(
                                                 \begin{array}{l}
                                                   0 \\
                                                   1 \\
                                                 \end{array}
                                               \right)\\
                                               0_{L^{1}} \\
                                             \end{array}
                                           \right) \\
                                               \end{array}
                                             \right)
\end{equation*}
with
\begin{equation*}
  \begin{array}{ccl}
    c_{2310} &=&-\tau_{k}\bigg\{\frac{1}{\xi+1}
    \bigg(\frac{e^{-i\omega_{k}}+e^{i\omega_{k}}}{\omega_{k}^{2}+\tau_{k}^{2}}
    -\frac{p_{+}}{(-i\mu\tau_{k}+\omega_{k})(i\tau_{k}+\omega_{k})}
    +\frac{p_{-}}{(i\omega_{k}-\mu\tau_{k})(i\omega_{k}+\tau_{k})}\bigg)\\
    &&-\frac{\xi}{(\xi+1)^{2}}
    \bigg(\frac{2}{\omega_{k}^{2}+\tau_{k}^{2}}
    -\frac{p_{+}e^{i\omega_{k}}}{(-i\mu\tau_{k}+\omega_{k})(i\tau_{k}+\omega_{k})}
    +\frac{p_{-}e^{-i\omega_{k}}}{(i\omega_{k}-\mu\tau_{k})(i\omega_{k}+\tau_{k})}\bigg)\bigg\}.\\
  \end{array}
\end{equation*}
%and
%\begin{equation*}
%  \begin{array}{ccl}
%    c_{2301} &=&\tau_{k}\big\{\frac{1}{\xi+1}\big[\frac{e^{-i\omega_{k}}}{(i\omega_{k}+\tau_{k})^{2}}+\frac{e^{i\omega_{k}}}{(\omega_{k}^{2}+\tau_{k}^{2})}\big]\\
%    &&-\frac{1}{i(\mu^{2}\tau_{k}^{2}+\omega_{k}^{2})(i\omega_{k}+\tau_{k})(\xi+1)}\big[\frac{(\eta\tau_{k}-i\omega_{k}-\tau_{k})(\omega_{k}+i\mu)}{i\omega_{k}+\tau_{k}}+
%    \frac{(\eta\tau_{k}+i\omega_{k}-\tau_{k})(\omega_{k}-i\mu)}{-i\omega_{k}+\tau_{k}}\big]\\
%    &&-\frac{1}{(\xi+1)^{2}}\big[\frac{(\eta\tau_{k}-i\omega_{k}-\tau_{k})\xi e^{i\omega_{k}}}{(i\omega_{k}+\mu\tau_{k})(i\omega_{k}+\tau_{k})(i\omega_{k}-\tau_{k})}
%    +\frac{(\eta\tau_{k}+i\omega_{k}-\tau_{k})\xi e^{-i\omega_{k}}}{(i\omega_{k}-\mu\tau_{k})(-i\omega_{k}+\tau_{k})(i\omega_{k}-\tau_{k})}-\frac{2\xi}{(i\omega_{k}+\tau_{k})(i\omega_{k}-\tau_{k})}\big]\big\}.\\
%  \end{array}
%\end{equation*}
Therefore, according to (\ref{lambdaIBtaukPih10}) and (\ref{lambdaIBtaukPih01}), we have
\begin{equation}\label{hate2hate3}
  \begin{array}{ccl}
L_{2}(\hat{e}_{2},\hat{e}_{3})& = & \frac{1}{2}c_{2310}(0-\mathcal{A}_{h})^{-1}\Pi_{h}\left(
                                             \begin{array}{c}
                                               0_{\mathbb{R}} \\
                                               \left(
                                                 \begin{array}{c}
                                                   \left(
                                                     \begin{array}{c}
                                                       1 \\
                                                       0 \\
                                                     \end{array}
                                                   \right)
                                                    \\
                                                   0_{L^{1}} \\
                                                 \end{array}
                                               \right)
                                                \\
                                             \end{array}
                                           \right)\\
&&-\frac{1}{2}c_{2310}(0-\mathcal{A}_{h})^{-1}\Pi_{h}\left(
                                             \begin{array}{c}
                                               0_{\mathbb{R}} \\
                                               \left(
                                                 \begin{array}{c}
                                                   \left(
                                                     \begin{array}{c}
                                                       0 \\
                                                       1 \\
                                                     \end{array}
                                                   \right)
                                                    \\
                                                   0_{L^{1}} \\
                                                 \end{array}
                                               \right)
                                                \\
                                             \end{array}
                                           \right)\\
 & = & \frac{1}{2}c_{2310}\left(
 \begin{array}{c}
 0_{\mathbb{R}} \\
 \left(0-B_{\tau}|_{\hat{\Pi}_{h}(X)}\right)^{-1}\hat{\Pi}_{h}\left(
 \begin{array}{c}
 \left(
 \begin{array}{c}
 1 \\
 0 \\
 \end{array}
 \right)\\
 0_{L^{1}} \\
 \end{array}
 \right) \\
 \end{array}
 \right)\\
 &&-\frac{1}{2}c_{2310}\left(
 \begin{array}{c}
 0_{\mathbb{R}} \\
 \left(0-B_{\tau}|_{\hat{\Pi}_{h}(X)}\right)^{-1}\hat{\Pi}_{h}\left(
 \begin{array}{c}
 \left(
 \begin{array}{c}
 0 \\
 1 \\
 \end{array}
 \right)\\
 0_{L^{1}} \\
 \end{array}
 \right) \\
 \end{array}
 \right)\\
 & = &\left(
        \begin{array}{c}
          0_{\mathbb{R}} \\
          \left(
            \begin{array}{c}
              0_{\mathbb{R}^{2}} \\
               \left(
                 \begin{array}{c}
                   \psi_{2,3,1} \\
                   \psi_{2,3,2} \\
                 \end{array}
               \right)
               \\
            \end{array}
          \right)
           \\
        \end{array}
      \right),
  \end{array}
\end{equation}
where
\begin{equation}\label{psi231}
  \begin{array}{ccl}
    \psi_{2,3,1}& = & \frac{1}{2}c_{2310}\bigg(\frac{1}{|\Delta(0)|}e^{-\tau_{k}\cdot}
     -\left(\frac{d}{d\lambda}|\Delta(i\omega_{k})|\right)^{-1}
     \frac{e^{-(i\omega_{k}+\tau_{k})\cdot}}{-i\omega_{k}} \\
     && -\left(\frac{d}{d\lambda}|\Delta(-i\omega_{k})|\right)^{-1}
       \frac{e^{-(-i\omega_{k}+\tau_{k})\cdot}}{i\omega_{k}}\bigg)\\
  \end{array}
\end{equation}
and
\begin{equation}\label{psi232}
  \begin{array}{ccl}
    \psi_{2,3,2}& = & \frac{1}{2}c_{2310}\bigg(\frac{1}{|\Delta(0)|}\frac{\eta\tau_{k}-\tau_{k}}{\tau_{k}}
    e^{-\mu\tau_{k}\cdot}
     -\left(\frac{d}{d\lambda}|\Delta(i\omega_{k})|\right)^{-1}
     \frac{p_{+}e^{-(i\omega_{k}+\mu\tau_{k})\cdot}}{-i\omega_{k}} \\
     && -\left(\frac{d}{d\lambda}|\Delta(-i\omega_{k})|\right)^{-1}
       \frac{p_{-}e^{-(-i\omega_{k}+\mu\tau_{k})\cdot}}{i\omega_{k}}\bigg).\\
  \end{array}
\end{equation}
\noindent \textbf{(iii) Computation of $L_{2}(\hat{e}_{3},\hat{e}_{3})$}: Since $\mathcal{A}_{c}\hat{e}_{3}=-i\omega_{k}\hat{e}_{3}$, the equation
\begin{equation*}
L_{2}(\mathcal{A}_{c}\hat{e}_{3},\hat{e}_{3})+L_{2}(\hat{e}_{3},\mathcal{A}_{c}\hat{e}_{3})= \mathcal{A}_{h}L_{2}(\hat{e}_{3},\hat{e}_{3})+\frac{1}{2!}\Pi_{h}D^{2}F(0)(\hat{e}_{3},\hat{e}_{3})
\end{equation*}
is equivalent to
\begin{equation*}
(-2i\omega_{k}-\mathcal{A}_{h})L_{2}(\hat{e}_{3},\hat{e}_{3})= \frac{1}{2!}\Pi_{h}D^{2}F(0)(\hat{e}_{3},\hat{e}_{3}),
\end{equation*}
where
\begin{equation*}
  \begin{array}{ccl}
     &&D^{2}F(0)(\hat{e}_{3},\hat{e}_{3})
     =  \left(
             \begin{array}{c}
               0_{\mathbb{R}} \\
               D^{2}W(0)\left(
                                          \begin{array}{c}
                                            0_{\mathbb{R}} \\
                                            \left(
                                              \begin{array}{c}
                                                0_{\mathbb{R}^{2}} \\
                                                \left(
                                                  \begin{array}{c}
                                                    b_{3} \\
                                                    b_{4} \\
                                                  \end{array}
                                                \right)
                                                 \\
                                              \end{array}
                                            \right)
                                             \\
                                          \end{array}
                                        \right)^{2}\\
             \end{array}
           \right)
      = \left(
                     \begin{array}{c}
                       0_{\mathbb{R}} \\
                       \tau_{k}D^{2}H(\overline{v}_{\tau_{k}})\left(
                                                                          \begin{array}{c}
                                                                            0_{\mathbb{R}^{2}} \\
                                                                            \left(
                                                  \begin{array}{c}
                                                    b_{3} \\
                                                    b_{4} \\
                                                  \end{array}
                                                \right) \\
                                                                          \end{array}
                                                                        \right)^{2}\\
                     \end{array}
                   \right).
      \\
  \end{array}
\end{equation*}
Thus, we have
\begin{equation*}
  D^{2}F(0)(\hat{e}_{3},\hat{e}_{3})=c_{3310}\left(
                                               \begin{array}{c}
                                                 0_{\mathbb{R}} \\
                                                 \left(
                                             \begin{array}{c}
                                                \left(
                                                 \begin{array}{l}
                                                   1 \\
                                                   0 \\
                                                 \end{array}
                                               \right)\\
                                               0_{L^{1}} \\
                                             \end{array}
                                           \right) \\
                                               \end{array}
                                             \right)
                                      -c_{3310}\left(
                                               \begin{array}{c}
                                                 0_{\mathbb{R}} \\
                                                 \left(
                                             \begin{array}{c}
                                                \left(
                                                 \begin{array}{l}
                                                   0 \\
                                                   1 \\
                                                 \end{array}
                                               \right)\\
                                               0_{L^{1}} \\
                                             \end{array}
                                           \right) \\
                                               \end{array}
                                             \right)
\end{equation*}
with
\begin{equation*}
  \begin{array}{ccl}
    c_{3310} & =&\frac{2\tau_{k}(\xi+1-\xi e^{i\omega_{k}})\left(e^{i\omega_{k}}(i\omega_{k}-\mu\tau)-p_{-}(i\omega_{k}-\tau)\right)}
    {(i\omega_{k}-\mu\tau_{k})(i\tau_{k}+\omega_{k})^{2}(\xi+1)^{2}}.\\
  \end{array}
\end{equation*}
Therefore, in accordance with (\ref{lambdaIBtaukPih10}) and (\ref{lambdaIBtaukPih01}), we have
\begin{equation}\label{hate3hate3}
  \begin{array}{ccl}
L_{2}(\hat{e}_{3},\hat{e}_{3})& = & \frac{1}{2}c_{3310}(-2i\omega_{k}-\mathcal{A}_{h})^{-1}\Pi_{h}\left(
                                             \begin{array}{c}
                                               0_{\mathbb{R}} \\
                                               \left(
                                                 \begin{array}{c}
                                                   \left(
                                                     \begin{array}{c}
                                                       1 \\
                                                       0 \\
                                                     \end{array}
                                                   \right)
                                                    \\
                                                   0_{L^{1}} \\
                                                 \end{array}
                                               \right)
                                                \\
                                             \end{array}
                                           \right)\\
&&-\frac{1}{2}c_{3310}(-2i\omega_{k}-\mathcal{A}_{h})^{-1}\Pi_{h}\left(
                                             \begin{array}{c}
                                               0_{\mathbb{R}} \\
                                               \left(
                                                 \begin{array}{c}
                                                   \left(
                                                     \begin{array}{c}
                                                       0 \\
                                                       1 \\
                                                     \end{array}
                                                   \right)
                                                    \\
                                                   0_{L^{1}} \\
                                                 \end{array}
                                               \right)
                                                \\
                                             \end{array}
                                           \right)\\
 & = & \frac{1}{2}c_{3310}\left(
 \begin{array}{c}
 0_{\mathbb{R}} \\
 \left(-2i\omega_{k}-B_{\tau}|_{\hat{\Pi}_{h}(X)}\right)^{-1}\hat{\Pi}_{h}\left(
 \begin{array}{c}
 \left(
 \begin{array}{c}
 1 \\
 0 \\
 \end{array}
 \right)\\
 0_{L^{1}} \\
 \end{array}
 \right) \\
 \end{array}
 \right)\\
 &&-\frac{1}{2}c_{3310}\left(
 \begin{array}{c}
 0_{\mathbb{R}} \\
 \left(-2i\omega_{k}-B_{\tau}|_{\hat{\Pi}_{h}(X)}\right)^{-1}\hat{\Pi}_{h}\left(
 \begin{array}{c}
 \left(
 \begin{array}{c}
 0 \\
 1 \\
 \end{array}
 \right)\\
 0_{L^{1}} \\
 \end{array}
 \right) \\
 \end{array}
 \right)\\
 & = &\left(
        \begin{array}{c}
          0_{\mathbb{R}} \\
          \left(
            \begin{array}{c}
              0_{\mathbb{R}^{2}} \\
               \left(
                 \begin{array}{c}
                   \psi_{3,3,1} \\
                   \psi_{3,3,2} \\
                 \end{array}
               \right)
               \\
            \end{array}
          \right)
           \\
        \end{array}
      \right),
  \end{array}
\end{equation}
where
\begin{equation}\label{psi331}
  \begin{array}{ccl}
    \psi_{3,3,1}& = & \frac{1}{2}c_{3310}\bigg(\frac{1}{|\Delta(-2i\omega_{k})|}e^{-(-2i\omega_{k}+\tau_{k})\cdot}
     -\left(\frac{d}{d\lambda}|\Delta(i\omega_{k})|\right)^{-1}
     \frac{e^{-(i\omega_{k}+\tau_{k})\cdot}}{-3i\omega_{k}} \\
     && -\left(\frac{d}{d\lambda}|\Delta(-i\omega_{k})|\right)^{-1}
       \frac{e^{-(-i\omega_{k}+\tau_{k})\cdot}}{-i\omega_{k}}\bigg)\\
  \end{array}
\end{equation}
and
\begin{equation}\label{psi332}
  \begin{array}{ccl}
    \psi_{3,3,2}& = & \frac{1}{2}c_{3310}\bigg(\frac{1}{|\Delta(-2i\omega_{k})|}\frac{\eta\tau_{k}-(-2i\omega_{k}+\tau_{k})}{-2i\omega_{k}+\tau_{k}}
    e^{-(-2i\omega_{k}+\mu\tau_{k})\cdot}
     -\left(\frac{d}{d\lambda}|\Delta(i\omega_{k})|\right)^{-1}
     \frac{p_{+}e^{-(i\omega_{k}+\mu\tau_{k})\cdot}}{-3i\omega_{k}} \\
     && -\left(\frac{d}{d\lambda}|\Delta(-i\omega_{k})|\right)^{-1}
       \frac{p_{-}e^{-(-i\omega_{k}+\mu\tau_{k})\cdot}}{-i\omega_{k}}\bigg).\\
  \end{array}
\end{equation}
By using a similar method, we obtain the following results:
\begin{equation}\label{L2L2e2e3}
  L_{2}(\hat{e}_{2},\hat{e}_{3})=L_{2}(\hat{e}_{3},\hat{e}_{2})=\left(
                                                                  \begin{array}{c}
                                                                    0_{\mathbb{R}} \\
                                                                    \left(
                                                                      \begin{array}{c}
                                                                        0_{\mathbb{R}^{2}} \\
                                                                        \left(
                                                                          \begin{array}{c}
                                                                            \psi_{2,3,1} \\
                                                                            \psi_{2,3,2} \\
                                                                          \end{array}
                                                                        \right)
                                                                         \\
                                                                      \end{array}
                                                                    \right)
                                                                     \\
                                                                  \end{array}
                                                                \right).
\end{equation}

\subsection{Normal Form of the reduced system}
\noindent

In this section, we will compute the normal form by applying the normal form theory developed in Liu et al. \cite{LiuZhihuaMagalPierreRuanShigui-JDE-2014} to system (\ref{dwtdt}). Define $\Theta_{m}^{c}: V^{m}(\mathcal{X}_{c},\mathcal{X}_{c})\rightarrow V^{m}(\mathcal{X}_{c},\mathcal{X}_{c})$ by
\begin{equation}\label{Thetamc}
  \begin{array}{cc}
    \Theta_{m}^{c}(G_{c}):=[\mathcal{A}_{c},G_{c}], & \forall G_{c}\in V^{m}(\mathcal{X}_{c},\mathcal{X}_{c}), \\
  \end{array}
\end{equation}
and $\Theta_{m}^{h}: V^{m}(\mathcal{X}_{c},\mathcal{X}_{h}\cap D(\mathcal{A}))\rightarrow V^{m}(\mathcal{X}_{c},\mathcal{X}_{h})$ by
\begin{equation}\label{Thetamh}
  \begin{array}{cc}
    \Theta_{m}^{h}(G_{h}):=[\mathcal{A},G_{h}], & \forall G_{h}\in V^{m}(\mathcal{X}_{c},\mathcal{X}_{h}\cap D(\mathcal{A})). \\
  \end{array}
\end{equation}
We decompose $V^{m}(\mathcal{X}_{c},\mathcal{X}_{c})$ into the direction sum
\begin{equation*}
  V^{m}(\mathcal{X}_{c},\mathcal{X}_{c})=\mathcal{R}_{m}^{c}\oplus\mathcal{C}_{m}^{c},
\end{equation*}
where
\begin{equation*}
  \mathcal{R}_{m}^{c}:=\mathrm{R}(\Theta_{m}^{c}),
\end{equation*}
is the range of $\Theta_{m}^{c}$, and $\mathcal{C}_{m}^{c}$ is some complementary space of $\mathcal{R}_{m}^{c}$ into $V^{m}(\mathcal{X}_{c},\mathcal{X}_{c})$. Define $\mathcal{P}_{m}: V^{m}(\mathcal{X}_{c},\mathcal{X})\rightarrow V^{m}(\mathcal{X}_{c},\mathcal{X})$ the bounded linear projector satisfying
\begin{equation*}
  \begin{array}{ccc}
     \mathcal{P}_{m}(V^{m}(\mathcal{X}_{c},\mathcal{X}))=\mathcal{R}_{m}^{c}\oplus V^{m}(\mathcal{X}_{c},\mathcal{X}_{h}) & \mbox{and} & (I-\mathcal{P}_{m})(V^{m}(\mathcal{X}_{c},\mathcal{X}))=\mathcal{C}_{m}^{c}. \\
  \end{array}
\end{equation*}
Now we apply the method described in [\cite{LiuZhihuaMagalPierreRuanShigui-JDE-2014} Theorem 4.4 for $k=3$] to system (\ref{dwtdt}). The main point is to compute $G_{2}\in V^{2}(\mathcal{X}_{c},D(\mathcal{A}))$ defined such that
\begin{equation}\label{mathcalAG2}
  \begin{array}{ccc}
    [\mathcal{A},G_{2}](w_{c})=\mathcal{P}_{2}[\frac{1}{2!}D^{2}F(0)(w_{c},w_{c})] & \mbox{for each} & w_{c}\in \mathcal{X}_{c}, \\
  \end{array}
\end{equation}
in order to obtain the normal form because the reduced system is the following
\begin{equation}\label{dwct/dt}
    \begin{array}{ccl}
      \frac{dw_{c}(t)}{dt} & = & \mathcal{A}_{c}w_{c}(t)+\frac{1}{2!}\Pi_{c}D^{2}F_{3}(0)(w_{c}(t),w_{c}(t)) \\
       & & +\frac{1}{3!}\Pi_{c}D^{3}F_{3}(0)(w_{c}(t),w_{c}(t),w_{c}(t))+R_{c}(w_{c}(t)) \\
    \end{array}
\end{equation}
where
\begin{equation}\label{PicD2}
  \begin{array}{ccl}
\frac{1}{2!}\Pi_{c}D^{2}F_{3}(0)(w_{c},w_{c}) & = & \frac{1}{2!}\Pi_{c}D^{2}F_{2}(0)(w_{c},w_{c}) \\
     & = & \frac{1}{2!}\Pi_{c}D^{2}F_{2}(w_{c},w_{c})-[\mathcal{A}_{c},\Pi_{c}G_{2}](w_{c})  \\
  \end{array}
\end{equation}
and
\begin{equation}\label{PicD3}
  \begin{array}{ccl}
\frac{1}{3!}\Pi_{c}D^{3}F_{3}(0)(w_{c},w_{c},w_{c}) & = & \frac{1}{3!}\Pi_{c}D^{3}F_{2}(0)(w_{c},w_{c},w_{c})-\Pi_{c}[\mathcal{A},G_{3}](w_{c}) \\
     & = & \frac{1}{3!}\Pi_{c}D^{3}F_{2}(0)(w_{c},w_{c},w_{c})-[\mathcal{A}_{c},\Pi_{c}G_{3}](w_{c}).  \\
  \end{array}
\end{equation}
Set
\begin{equation*}
     \begin{array}{cc}
       G_{m,k}:=\Pi_{k}G_{m}, & \forall k=c,h,m\geq 2. \\
     \end{array}
\end{equation*}
Recall that (\ref{mathcalAG2}) is equivalent to finding $G_{2,c}\in V^{2}(\mathcal{X}_{c},\mathcal{X}_{c})$ and $G_{2,h}\in V^{2}(\mathcal{X}_{c},\mathcal{X}_{h}\cap D(\mathcal{A}))$ satisfying
\begin{equation}\label{mathcalAcG2c}
  [\mathcal{A}_{c},G_{2,c}]=\Pi_{c}\mathcal{P}_{2}[\frac{1}{2!}D^{2}F(0)(w_{c},w_{c})]
\end{equation}
and
\begin{equation}\label{mathcalAG2h}
  [\mathcal{A},G_{2,h}]=\Pi_{h}\mathcal{P}_{2}[\frac{1}{2!}D^{2}F(0)(w_{c},w_{c})].
\end{equation}
From (\ref{PicD3}), we know that the third order term $\frac{1}{3!}\Pi_{c}D^{3}F_{2}(0)(w_{c},w_{c},w_{c})$ in the equation is needed after computing the normal form up to the second order. In the following lemma we find the expression of $\frac{1}{3!}\Pi_{c}D^{3}F_{2}(0)(w_{c},w_{c},w_{c})$.
Set
\begin{equation*}
  w_{c}=\hat{\tau}\hat{e}_{1}+x_{1}\hat{e}_{2}+x_{2}\hat{e}_{3}.
\end{equation*}
We shall compute the normal form expressed in terms of the basis $\{\hat{e}_{1},\hat{e}_{2},\hat{e}_{3}\}$. Consider $V^{m}(\mathbb{C}^{3},\mathbb{C}^{3})$ and $V^{m}(\mathbb{C}^{3},\mathcal{X}_{h}\cap D(\mathcal{A}))$, which denote the linear space of the homogeneous polynomials of degree $m$ in $3$ real variables, $\hat{\tau}, x=(x_{1},x_{2})$ with coefficients in $\mathbb{C}^{3}$ and $\mathcal{X}_{h}\cap D(\mathcal{A})$, respectively. The operators $\Theta_{m}^{c}$ and $\Theta_{m}^{h}$ considered in (\ref{Thetamc}) and (\ref{Thetamh}) now act in the spaces $V^{m}(\mathbb{C}^{3},\mathbb{C}^{3})$ and $V^{m}(\mathbb{C}^{3},\mathcal{X}_{h}\cap D(\mathcal{A}))$, respectively, and satisfy
\begin{equation*}
    \begin{array}{ccl}
      \Theta_{m}^{c}(G_{m,c})\left(
                               \begin{array}{c}
                                 \hat{\tau} \\
                                 x_{1} \\
                                 x_{2} \\
                               \end{array}
                             \right)
       & = & [\mathcal{A}_{c},G_{m,c}]\left(
                               \begin{array}{c}
                                 \hat{\tau} \\
                                 x_{1} \\
                                 x_{2} \\
                               \end{array}
                             \right)=DG_{m,c}\mathcal{A}_{c}\left(
                               \begin{array}{c}
                                 \hat{\tau} \\
                                 x_{1} \\
                                 x_{2} \\
                               \end{array}
                             \right)-\mathcal{A}_{c}G_{m,c}\left(
                               \begin{array}{c}
                                 \hat{\tau} \\
                                 x_{1} \\
                                 x_{2} \\
                               \end{array}
                             \right) \\
       & = & \left(
               \begin{array}{c}
                 D_{x}G_{m,c}^{1}\left(
                               \begin{array}{c}
                                 \hat{\tau} \\
                                 x_{1} \\
                                 x_{2} \\
                               \end{array}
                             \right)M_{c}x \\
                 D_{x}\left(
                        \begin{array}{c}
                          G_{m,c}^{2} \\
                          G_{m,c}^{3} \\
                        \end{array}
                      \right)\left(
                               \begin{array}{c}
                                 \hat{\tau} \\
                                 x_{1} \\
                                 x_{2} \\
                               \end{array}
                             \right)M_{c}\left(
                                           \begin{array}{c}
                                             x_{1} \\
                                             x_{2} \\
                                           \end{array}
                                         \right)-M_{c}\left(
                        \begin{array}{c}
                          G_{m,c}^{2} \\
                          G_{m,c}^{3} \\
                        \end{array}
                      \right)\left(
                               \begin{array}{c}
                                 \hat{\tau} \\
                                 x_{1} \\
                                 x_{2} \\
                               \end{array}
                             \right)\\
               \end{array}
             \right),
        \\
    \end{array}
\end{equation*}
\begin{equation*}
  \Theta_{m}^{h}(G_{m,h})=[\mathcal{A},G_{m,h}]=DG_{m,h}\mathcal{A}_{c}-\mathcal{A}_{h}G_{m,h},
\end{equation*}
\begin{equation}\label{Gmc}
    \begin{array}{cc}
      \forall G_{m,c}\left(
                       \begin{array}{c}
                         \hat{\tau} \\
                         x_{1} \\
                         x_{2} \\
                       \end{array}
                     \right)=\left(
                               \begin{array}{c}
                                 G_{m,c}^{1} \\
                                 G_{m,c}^{2} \\
                                 G_{m,c}^{3} \\
                               \end{array}
                             \right)\left(
                       \begin{array}{c}
                         \hat{\tau} \\
                         x_{1} \\
                         x_{2} \\
                       \end{array}
                     \right)\in V^{m}(\mathbb{C}^{3},\mathbb{C}^{3}),
       & \forall G_{m,h}\in V^{m}(\mathbb{C}^{3},\mathcal{X}_{h}\cap D(\mathcal{A})) \\
    \end{array}
\end{equation}
with
\begin{equation*}
  \begin{array}{ccc}
    \mathcal{A}_{c}=\left(
                      \begin{array}{ccc}
                        0 & 0 & 0 \\
                        0 & i\omega_{k} & 0 \\
                        0 & 0 & -i\omega_{k} \\
                      \end{array}
                    \right)
     & \mbox{and} & M_{c}=\left(
                     \begin{array}{cc}
                       i\omega_{k} & 0 \\
                       0 & -i\omega_{k} \\
                     \end{array}
                   \right).
      \\
  \end{array}
\end{equation*}
We define $\overline{\Theta}_{m}^{c}: V^{m}(\mathbb{C}^{3},\mathbb{C}^{2})\rightarrow V^{m}(\mathbb{C}^{3},\mathbb{C}^{2})$ by
\begin{equation}\label{overlineThetamc}
  \begin{array}{ccl}
      \overline{\Theta}_{m}^{c}\left(
                        \begin{array}{c}
                          G_{m,c}^{2} \\
                          G_{m,c}^{3} \\
                        \end{array}
                      \right)\left(
                       \begin{array}{c}
                         \hat{\tau} \\
                         x_{1} \\
                         x_{2} \\
                       \end{array}
                     \right) & = & D_{x}\left(
                        \begin{array}{c}
                          G_{m,c}^{2} \\
                          G_{m,c}^{3} \\
                        \end{array}
                      \right)\left(
                       \begin{array}{c}
                         \hat{\tau} \\
                         x_{1} \\
                         x_{2} \\
                       \end{array}
                     \right)M_{c}\left(
                                    \begin{array}{c}
                                      x_{1} \\
                                      x_{2} \\
                                    \end{array}
                                  \right) \\
     &  & -M_{c}\left(
                        \begin{array}{c}
                          G_{m,c}^{2} \\
                          G_{m,c}^{3} \\
                        \end{array}
                      \right)\left(
                       \begin{array}{c}
                         \hat{\tau} \\
                         x_{1} \\
                         x_{2} \\
                       \end{array}
                     \right),
                      \forall \left(
                        \begin{array}{c}
                          G_{m,c}^{2} \\
                          G_{m,c}^{3} \\
                        \end{array}
                      \right)\in V^{m}(\mathbb{C}^{3},\mathbb{C}^{2}). \\
  \end{array}
\end{equation}
From Lemma 5.19 in \cite{LiuZhihuaMagalPierreRuanShigui-JDE-2014}, we have
\begin{equation}\label{NoverlineTheta23}
    \begin{array}{ccl}
      N(\overline{\Theta}_{2}^{c}) & = & \mbox{span}\left\{\left(
                                                   \begin{array}{c}
                                                     x_{1}\hat{\tau} \\
                                                     0 \\
                                                   \end{array}
                                                 \right),
                                                 \left(
                                                   \begin{array}{c}
                                                     0 \\
                                                     x_{2}\hat{\tau} \\
                                                   \end{array}
                                                 \right)
    \right\}, \\
    R(\overline{\Theta}_{2}^{c})& = &\mbox{span}\left\{
                                                   \begin{array}{c}
                                                     \left(
                                                   \begin{array}{c}
                                                     x_{1}x_{2} \\
                                                     0 \\
                                                   \end{array}
                                                 \right), \left(
                                                   \begin{array}{c}
                                                     x_{1}^{2} \\
                                                     0 \\
                                                   \end{array}
                                                 \right),\left(
                                                   \begin{array}{c}
                                                     x_{2}\hat{\tau} \\
                                                     0 \\
                                                   \end{array}
                                                 \right),\left(
                                                   \begin{array}{c}
                                                     x_{2}^{2} \\
                                                     0 \\
                                                   \end{array}
                                                 \right),\left(
                                                   \begin{array}{c}
                                                     \tau^{2} \\
                                                     0 \\
                                                   \end{array}
                                                 \right) \\
                                                     \left(
                                                   \begin{array}{c}
                                                     0 \\
                                                     x_{1}x_{2} \\
                                                   \end{array}
                                                 \right),\left(
                                                   \begin{array}{c}
                                                     0 \\
                                                     x_{1}\hat{\tau} \\
                                                   \end{array}
                                                 \right),\left(
                                                   \begin{array}{c}
                                                     0 \\
                                                     x_{1}^{2} \\
                                                   \end{array}
                                                 \right),\left(
                                                   \begin{array}{c}
                                                     0 \\
                                                     x_{2}^{2} \\
                                                   \end{array}
                                                 \right),\left(
                                                   \begin{array}{c}
                                                     0 \\
                                                     \tau^{2} \\
                                                   \end{array}
                                                 \right) \\
                                                   \end{array}
  \right\},\\
      N(\overline{\Theta}_{3}^{c}) & = & \mbox{span}\left\{\left(
                                                   \begin{array}{c}
                                                     x_{1}^{2}x_{2} \\
                                                     0 \\
                                                   \end{array}
                                                 \right),
                                                 \left(
                                                   \begin{array}{c}
                                                     x_{1}\hat{\tau}^{2} \\
                                                     0 \\
                                                   \end{array}
                                                 \right),
                                                 \left(
                                                   \begin{array}{c}
                                                     0 \\
                                                     x_{1}x_{2}^{2} \\
                                                   \end{array}
                                                 \right),
                                                 \left(
                                                   \begin{array}{c}
                                                     0 \\
                                                     x_{2}\hat{\tau}^{2} \\
                                                   \end{array}
                                                 \right)
    \right\}, \\
   R(\overline{\Theta}_{3}^{c}) & = & \mbox{span}\left\{
                                                 \begin{array}{c}
                                                   \left(
                                                      \begin{array}{c}
                                                        x_{1}x_{2}\hat{\tau} \\
                                                        0 \\
                                                      \end{array}
                                                    \right), \left(
                                                      \begin{array}{c}
                                                        x_{1}x_{2}^{2} \\
                                                        0 \\
                                                      \end{array}
                                                    \right),\left(
                                                      \begin{array}{c}
                                                        x_{1}^{2}\hat{\tau} \\
                                                        0 \\
                                                      \end{array}
                                                    \right),\left(
                                                      \begin{array}{c}
                                                        x_{1}^{3} \\
                                                        0 \\
                                                      \end{array}
                                                    \right)\\
                                                   \left(
                                                      \begin{array}{c}
                                                        x_{2}\hat{\tau}^{2} \\
                                                        0 \\
                                                      \end{array}
                                                    \right), \left(
                                                      \begin{array}{c}
                                                        x_{2}^{2}\hat{\tau} \\
                                                        0 \\
                                                      \end{array}
                                                    \right),\left(
                                                      \begin{array}{c}
                                                        x_{2}^{3} \\
                                                        0 \\
                                                      \end{array}
                                                    \right),\left(
                                                      \begin{array}{c}
                                                        \hat{\tau}^{3} \\
                                                        0 \\
                                                      \end{array}
                                                    \right) \\
                                                   \left(
                                                      \begin{array}{c}
                                                        0 \\
                                                        x_{1}x_{2}\hat{\tau} \\
                                                      \end{array}
                                                    \right), \left(
                                                      \begin{array}{c}
                                                        0 \\
                                                        x_{1}\hat{\tau}^{2} \\
                                                      \end{array}
                                                    \right), \left(
                                                      \begin{array}{c}
                                                        0 \\
                                                        x_{1}^{2}x_{2} \\
                                                      \end{array}
                                                    \right), \left(
                                                      \begin{array}{c}
                                                        0 \\
                                                        x_{1}^{2}\hat{\tau} \\
                                                      \end{array}
                                                    \right) \\
                                                   \left(
                                                      \begin{array}{c}
                                                        0 \\
                                                        x_{1}^{3} \\
                                                      \end{array}
                                                    \right),\left(
                                                      \begin{array}{c}
                                                        0 \\
                                                        x_{2}^{2}\hat{\tau} \\
                                                      \end{array}
                                                    \right), \left(
                                                      \begin{array}{c}
                                                        0 \\
                                                        x_{2}^{3} \\
                                                      \end{array}
                                                    \right), \left(
                                                      \begin{array}{c}
                                                        0 \\
                                                        \hat{\tau}^{3} \\
                                                      \end{array}
                                                    \right) \\
                                                 \end{array}
\right\}.\\
    \end{array}
\end{equation}
Define $P_{m}^{R}$ and $P_{m}^{N}: V^{m}(\mathbb{C}^{3},\mathbb{C}^{2})\rightarrow V^{m}(\mathbb{C}^{3},\mathbb{C}^{2})$ the bounded linear projectors satisfying
\begin{equation*}
  P_{m}^{R}(V^{m}(\mathbb{C}^{3},\mathbb{C}^{2}))=R(\overline{\Theta}_{m}^{c})
\end{equation*}
and
\begin{equation*}
  P_{m}^{N}(V^{m}(\mathbb{C}^{3},\mathbb{C}^{2}))=N(\overline{\Theta}_{m}^{c}).
\end{equation*}
We are now ready to compute the normal form of the reduced system expressed in terms of the basis $\left\{\hat{e}_{1},\hat{e}_{2},\hat{e}_{3}\right\}$ of $\mathcal{X}_{c}$. From (\ref{PicD2F0}), (\ref{PihD2F0}), (\ref{mathcalAcG2c}), (\ref{mathcalAG2h}), (\ref{Gmc}) and (\ref{NoverlineTheta23}), we know that to find $G_{2}\in V^{2}(\mathcal{X}_{c},D(\mathcal{A}))$ defined in (\ref{mathcalAG2}) is equivalent to finding $G_{2,c}=\left(

\end{equation*}
\noindent
From (\ref{PicD2}), it is easy to obtain the second order terms of the normal form expressed in terms of the basis $\{\hat{e}_{1},\hat{e}_{2},\hat{e}_{3}\}$:
\begin{equation*}
    %
\end{equation*}
Notice that the terms $o(|x|\tau^{2})$ are irrelevant to determine the generic Hopf bifurcation. Hence, it is only needed to compute the coefficients of
\begin{equation*}
    %
\end{equation*}
in the third order terms of the normal form.
\noindent
Firstly from (\ref{hate2hate2}), (\ref{hate2hate3}), and (\ref{hate3hate3}), we obtain
\begin{equation*}
    %
         \right)
\end{equation*}
Therefore, according to (\ref{D2W0w1w2}), we have
\begin{equation*}
  [D^{2}W(0)(w_{c},G_{2}(w_{c}))]_{\hat{\tau}=0}=\tau_{k}D^{2}H(\overline{v}_{\tau_{k}})(\hat{v}_{c},G_{2}(\hat{v}_{c}))=\left(
                                                                                                               %
\end{equation*}
Hence, according to (\ref{hatPic10}) and (\ref{hatPic01}), we have
\begin{equation}\label{hatPicD2}
    %
\end{equation*}
Now we obtain the third order terms of the normal form expressed in terms of the basis $\{\hat{e}_{1},\hat{e}_{2},\hat{e}_{3}\}$, namely,
\begin{equation*}
    %
\end{equation*}
Therefore we obtain the following normal form of the reduced system
\begin{equation*}
  \frac{d}{dt}\left(
                \begin{array}{c}
                  x_{1}(t) \\
                  x_{2}(t) \\
                \end{array}
              \right)=M_{c}\left(
                \begin{array}{c}
                  x_{1}(t) \\
                  x_{2}(t) \\
                \end{array}
              \right)+\left(
                        \begin{array}{c}
                          A_{1}x_{1}\hat{\tau} \\
                          \overline{A}_{1}x_{2}\hat{\tau} \\
                        \end{array}
                      \right)+\left(
                                \begin{array}{c}
                                  C_{1}x_{1}^{2}x_{2} \\
                                  \overline{C}_{1}x_{1}x_{2}^{2} \\
                                \end{array}
                              \right)+o(|x|\hat{\tau}^{2}+|(\hat{\tau},x)|^{4}).
\end{equation*}
The normal form above can be written in real coordinates $(w_{1},w_{2})$ through the change of variables $x_{1}=w_{1}-iw_{2}, x_{2}=w_{1}-iw_{2}$. Setting $w_{1}=\rho \cos(\xi), w_{2}=\rho \sin(\xi)$, this normal form becomes
\begin{equation}\label{rhoxi}
  \left\{
    \begin{array}{l}
      \dot{\rho}=\iota_{1}\hat{\tau}\rho+\iota_{2}\rho^{3}+o(\hat{\tau}^{2}\rho+|(\rho, \hat{\tau})|^{4}), \\
      \dot{\xi}=-\sigma_{k}+o(|(\rho,\hat{\tau}|), \\
    \end{array}
  \right.
\end{equation}
where
\begin{equation*}
     \begin{array}{cc}
       \iota_{1}=Re(A_{1}), & \iota_{2}=Re(C_{1}). \\
     \end{array}
\end{equation*}
Following Chow and Hale \cite{ShuiNeeChowJackKHale-1982} we know that the sign of $\iota_{1}\iota_{2}$ determines the direction of the bifurcation and that the sign of $\iota_{2}$ determines the stability of the nontrivial periodic orbits. In summary we have the following theorem.
\begin{theorem}
The flow of (\ref{system}) on the center manifold of the positive equilibrium at $\tau=\tau_{k}, k\in \mathbb{N}^{+}$, is given by (\ref{rhoxi}). Furthermore, we have the following
\begin{itemize}
  \item [(i)] Hopf bifurcation is supercritical if $\iota_{1}\iota_{2}<0$, and subcritical if $\iota_{1}\iota_{2}>0$;
  \item [(ii)] the nontrivial periodic solution is stable if $\iota_{2}<0$, and unstable if $\iota_{2}>0$.
\end{itemize}
\end{theorem}

\section{Numerical simulations and conclusions}

\noindent

In this section, we perform some numerical simulations to illustrate the
results showed in Theorem \ref{HopfBifurcation}. We choose the parameter
values in model (\ref{system}) based on the conditions of Hopf bifurcation
(see Figure 1). By calculation, we can easily yield that the first
bifurcation critical value is $\tau_{0}=23.2282 $. Figure (a) and Figure (b)
describe the evolution of $S(t)$ and $i(t,a)$. The solution curves of system
(\ref{system}) performs a sustained periodic oscillation at $\tau=24>\tau_{0}
$. As is shown in Figure (c), when the number of susceptible individuals
tends to 0, the number of infectious individuals reaches its maximum and
vice versa. In addition, the phase trajectories of $S(t)$ and $i(t,\cdot)$
consistently approaches the stable limit cycles. Figure (d) demonstrates the
change of $i(t,a)$ as time and infection-age vary at $\tau=24>\tau_{0}$.

The bifurcation parameter $\tau $ has essentially a huge impact on the
dynamical behavior of system (\ref{system}). Biologically speaking, for
smaller $\tau $, the stability of the unique positive equilibrium of system (%
\ref{system}) is not affected. However, when bifurcation parameter $\tau $
crosses the bifurcation critical value $\tau _{0}$, the sustained periodic
oscillation phenomenon appears around the positive equilibrium. In
conclusion, age structure induces the appearance of periodic solutions.
\begin{figure}[htbp]
\setlength{\belowcaptionskip}{1pt}
 \centering
\begin{minipage}[c]{0.4\textwidth}
    \centering
    \subfigure []
{\includegraphics[width=2.5in]{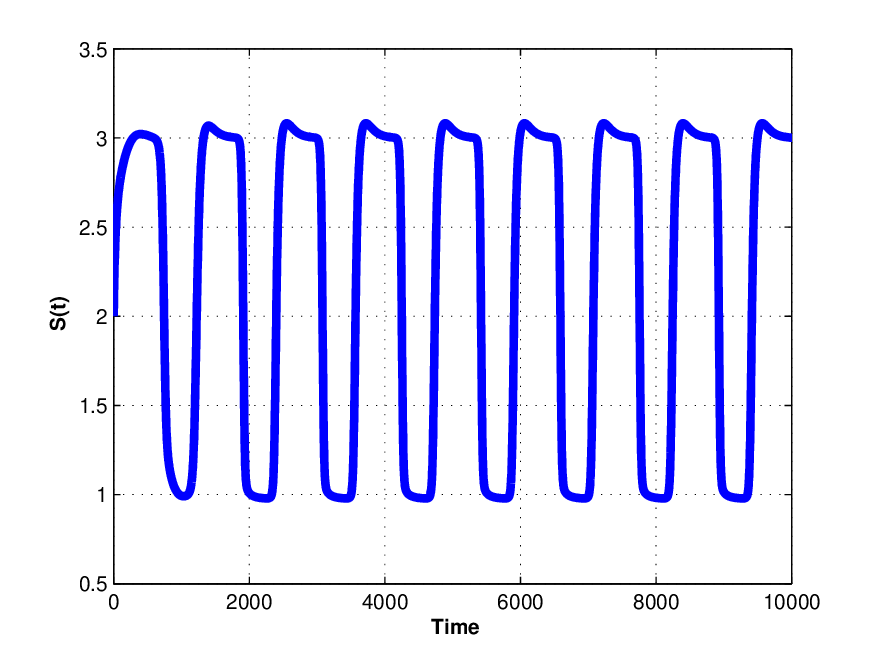}} \hfill
  \end{minipage}
\begin{minipage}[c]{0.4\textwidth}
    \centering
    \subfigure []
{\includegraphics[width=2.5in]{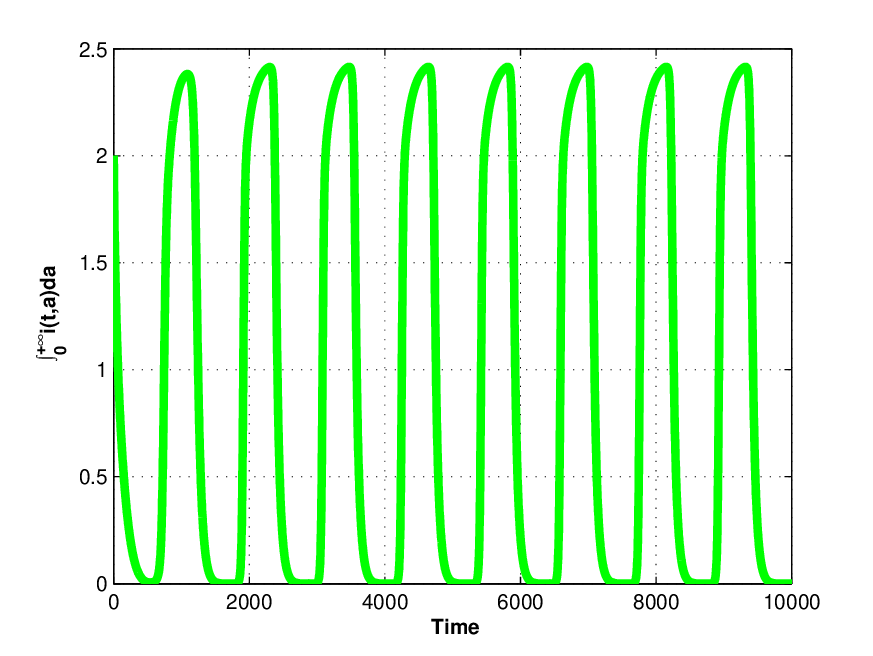}}
  \end{minipage}
\begin{minipage}[c]{0.4\textwidth}
    \centering
    \subfigure []
{\includegraphics[width=2.5in]{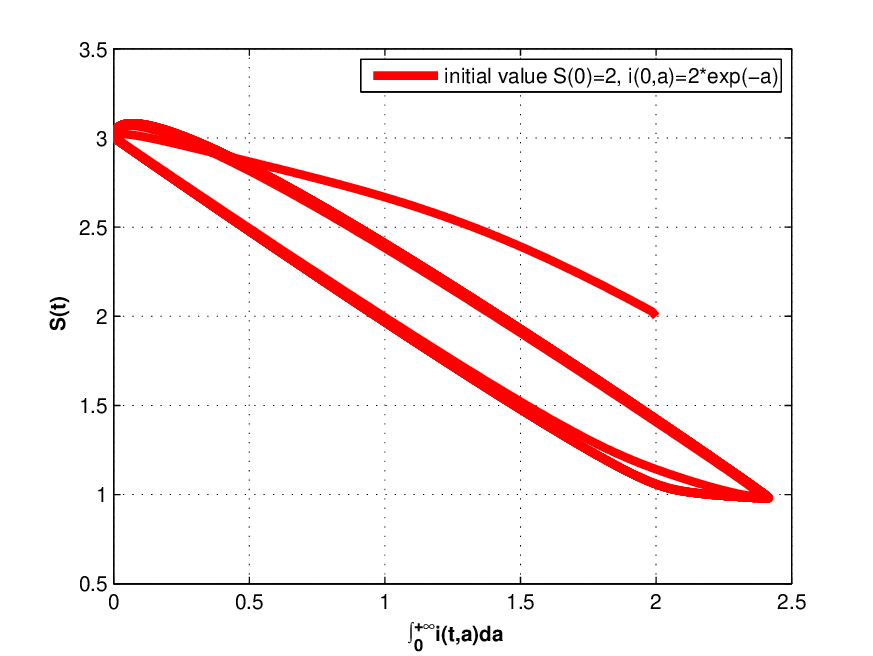}} \hfill
  \end{minipage}
\begin{minipage}[c]{0.4\textwidth}
    \centering
    \subfigure []
{\includegraphics[width=2.5in]{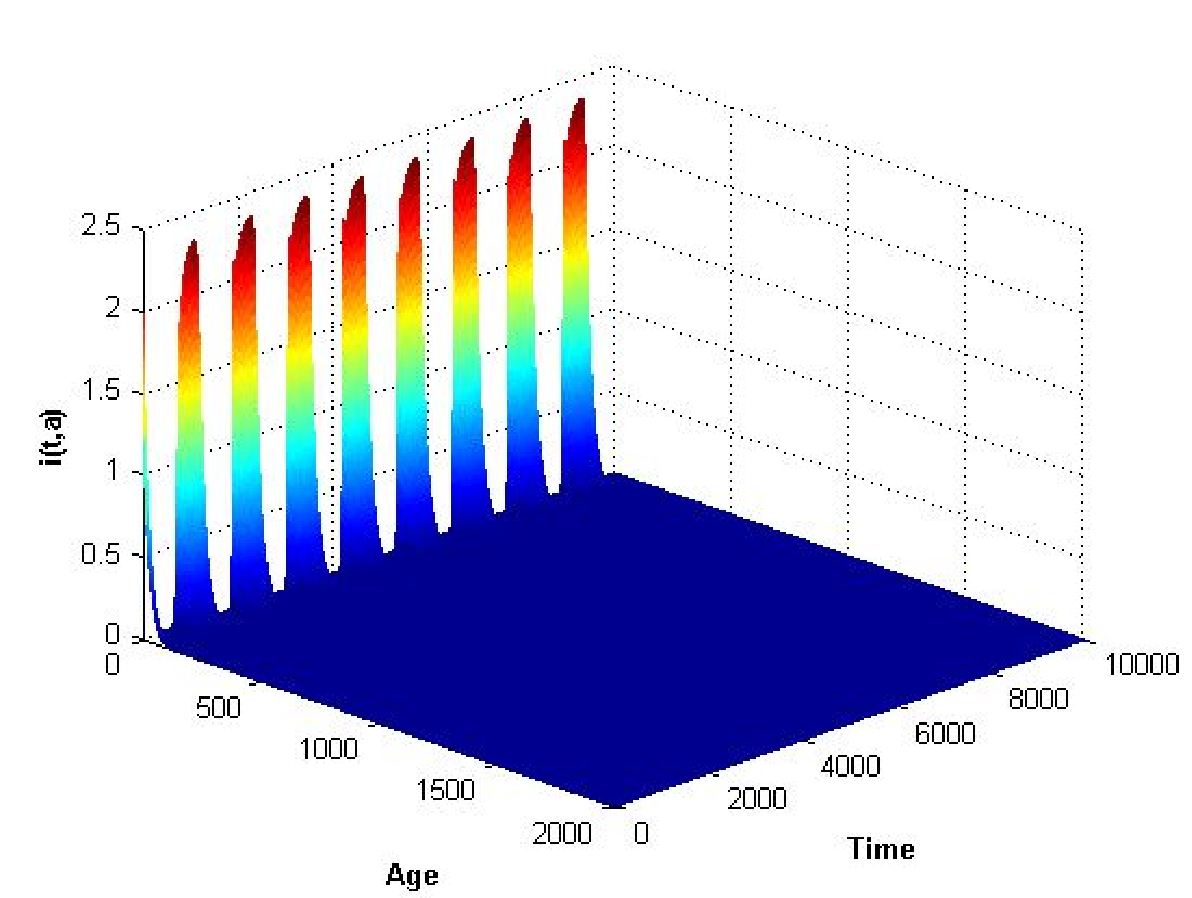}} \hfill
  \end{minipage}
\renewcommand{\figurename}{\footnotesize{Figure}}
\caption{\footnotesize{$A=0.6, \mu=0.2, \eta=0.81, \tau=24, \beta^{*}=e^{\tau}, A-\mu=0.4>0, E-C=0.0051>0, BD-2E=0.1551>0, B^2-2C=0.3065>0, BCD-E(B^2-2C)=0.0034>0$.}}
\end{figure}

\end{document}